\documentclass[11pt]{amsart}

\usepackage[utf8]{inputenc}
\usepackage[T1]{fontenc}

\usepackage{hyperref}

\usepackage[all]{xy}
\usepackage{amssymb}
\usepackage{mathdots}

\newtheorem{theorem}{Theorem}[section]
\newtheorem*{theorem*}{Theorem}
\newtheorem{lemma}[theorem]{Lemma}
\newtheorem{question}[theorem]{Question}
\newtheorem{corollary}[theorem]{Corollary}
\newtheorem*{corollary*}{Corollary}
\newtheorem{proposition}[theorem]{Proposition}
\newtheorem*{proposition*}{Proposition}
\newtheorem{claim}[theorem]{Claim}
\newtheorem*{claim*}{Claim}

\newtheorem{conjecture}[theorem]{Conjecture}

\theoremstyle{definition}
\newtheorem*{definition*}{Definition}
\newtheorem{definition}[theorem]{Definition}
\newtheorem{example}[theorem]{Example}

\theoremstyle{remark}

\newtheorem*{remark*}{Remark}
\newtheorem*{notation}{Notation}

\theoremstyle{plain}

\newcommand{\Z}{{\mathbb Z}}
\newcommand{\T}{{\mathbb T}}

\newcommand{\R}{{\mathbb R}}
\newcommand{\N}{{\mathbb N}}
\newcommand{\E}{{\mathbb E}}

\newcommand{\CC}{{\mathcal C}}

\newcommand{\CF}{{\mathcal F}}

\newcommand{\CB}{{\mathcal B}}

\newcommand{\CU}{{\mathcal U}}

\newcommand{\one}{{\mathbf 1}}

\newcommand{\inv}{^{-1}}
\newcommand{\id}{{\bf Id}}
\newcommand{\norm}[1]{\Vert #1\Vert}
\newcommand{\wt}{\widetilde}

\newcommand{\ve}{\varepsilon}

\newcommand{\bohrzero}{{\rm Bohr${}_0$ }}

\newcommand{\bohrzeroo}{{\rm Bohr${}_0^*$ }}

\newcommand{\nilbohrzero}[1]{{\rm Nil${}_{{#1}}$-Bohr${}_0$ }}
\newcommand{\nilbohrzeroo}[1]{{\rm Nil${}_{{#1}}$-Bohr${}_0^*$ }}

\newcommand{\nildbohrzero}{{\rm Nil${}_s$-Bohr${}_0$ }}
\newcommand{\nildbohrzeroo}{{\rm Nil${}_s$-Bohr${}_0^*$ }}
\newcommand{\RP}{\mathrm{RP}}

\newcommand{\RPs}{{\RP^{[s]}}}
\DeclareMathOperator{\ip}{IP}

\newcommand{\wh}{\widehat}

\newcommand{\myar}{\ar@[|(3)]}

\DeclareMathOperator{\range}{Range}

\begin{document}
\title{Variations on topological recurrence}
\author{Bernard Host}
\address{Universit\'e de Paris-Est Marne la Vall\'ee, Laboratoire d'analyse et de math\'ematiques appliqu\'{e}es CNRS UMR 8050\\
5 Bd. Descartes, Champs sur Marne\\
77454 Marne la Vall\'ee Cedex 2, France}
\email{bernard.host@univ-mlv.fr}

\author{Bryna Kra}
\address{ Department of Mathematics, Northwestern University \\ 2033 Sheridan Road Evanston \\ IL 60208-2730, USA} 
\email{kra@math.northwestern.edu}

\author{Alejandro Maass}
\address{Department of Mathematical Engineering \& 
Center for Mathematical Modeling UMI 2071 UCHILE-CNRS, University of Chile \\ Blanco Encalada 2120 \\ Santiago, Chile.}
\email{amaass@dim.uchile.cl} 

\thanks{The  second author was partially supported by NSF grant $1200971$ and the third author was partially supported by the B\'ezout Chair of the Universit\'e Paris-Est Marne-la-Vall\'ee.}

\begin{abstract}
Recurrence properties of systems and associated sets of integers that suffice for recurrence are classical objects in topological dynamics.   We describe relations between recurrence in different sorts of systems, study ways to formulate finite versions of recurrence, and describe connections to combinatorial problems.  In particular, we show that sets of Bohr recurrence (meaning sets of recurrence for rotations) suffice for recurrence in nilsystems.  Additionally, we prove an extension of this property for  multiple recurrence  in affine systems.
\end{abstract}
\maketitle

\section{Topological recurrence}

Van der Waerden's classic theorem~\cite{vdw} states that any finite coloring of the integers contains arbitrarily long monochromatic progressions. This has led to numerous refinements and strengthenings, 
with some of these obtained via the deep connections to topological dynamics introduced with the proof of Furstenberg and Weiss~\cite{FW}.  A direction that has been studied extensively is what restrictions can be placed on the step in the arithmetic progression, and in dynamics this corresponds to what sets arise as sets of recurrence.   Recurrence properties of systems and the associated sets of recurrence are classical notions both in topological dynamics and in additive combinatorics, and have numerous classically equivalent characterizations.  

Part of this article is a review of these connections, many of which are scattered throughout the literature, and we point out numerous open questions. 
Part of this article is new, particularly connections to objects that have recently shown to play a role 
in topological dynamics, such as nilsystems.  For both single and multiple recurrence, the class of nilsystems (see Section~\ref{sec:lift-Bohr} for definitions) plays a natural role.  This is reflected in work in the ergodic context on multiple convergence along arithmetic progressions~\cite{HK}.  In the topological context, a higher order regionally proximal relation was introduced in~\cite{HKM}, where the connection to nilsystems was made.  Further deep connections between these notions and that of topological recurrence were made in~\cite{HSY}.  Nilsystems have also been used to construct explicit examples of sets of multiple recurrence, for example in the work of~\cite{FLW, HSY}.  Thus the relation between recurrence and its connections with nilsystems have become a natural direction for further study.

Our main focus is how to formulate finite versions of recurrence related to van der Waerden's Theorem. One way is to 
fix a length for the progressions and then characterize the sets of recurrence for this fixed length.  
We then study classifying dynamical systems by their recurrence properties along arithmetic progressions of this length, 
seeking necessary or sufficient conditions for such recurrence.  In various guises, this problem has been studied by dynamicists and we consider this point of view in Section~\ref{sec:recurrence}.

In particular, we study a question asked by Katznelson~\cite{katznelson}: if $R$ is a set of recurrence for all rotations, is it a set of recurrence for all minimal topological dynamical systems? (See Section~\ref{sec:recurrence-for-families} for the definitions.)  
We give a partial answer to this question, showing that it holds when one restricts to the class of nilsystems (Theorem~\ref{th:rec-nil}) and its almost proximal extensions. We then turn to the similar questions for multiple recurrence.  In this setting, we show (Theorem~\ref{th:lift-affine}) that if $R$ is a set of $s$-recurrence for $s$-step affine nilsystems, then it is also a set of $t$-recurrence for all $t\geq s$ for the same class of systems. A summary of these implications is given in Figure 1.  

A second way to finitize van der Waerden's Theorem is by fixing the number of colors and studying the associated sets of recurrence.  This point of view has largely been ignored by dynamicists and we take this approach 
in Section~\ref{sec:large}, where we mainly pose further directions for study.

Throughout this article, we assume that $(X,T)$ denotes  a {\em (topological dynamical) system}, meaning that  $X$ is a compact metric space 
and $T\colon X\to X$ is a homeomorphism.  
While our primary focus is on topological recurrence, there are also measure theoretic analogs, where the underlying space is a probability measure preserving system $(X, \CB, \mu)$ endowed with a measurable, measure preserving transformation $T\colon X\to X$.  
Combinatorially, this corresponds to Szemer\'edi's Theorem and the connection to ergodic theory has 
been well studied.  While the measure theoretic and topological settings give rise to similar 
results, there are some differences and we point out some of the known measure theoretic analogs and  
pose some related questions.
 
\section{Variations on recurrence}
\label{sec:recurrence}

\subsection{Single recurrence}

Throughout, we focus on {\em minimal} systems $(X,T)$, meaning that no proper closed subset of $X$ is $T$-invariant.

\begin{definition}
We say that $R\subset \N$  is a {\em set of (topological) recurrence} if for every minimal system $(X,T)$ and every nonempty open set $U\subset  X$, there exists $n\in R$ such that $U\cap T^{-n}U\neq\emptyset$. 
\end{definition}

\begin{notation}
If $x\in X$ and $U\subset  X$ is an open set, we write 
$$
N(x,U)=\{n\in\N\colon T^nx\in U\}
$$
for the return times of the point $x$ to the neighborhood $U$ and
$$
N(U)=\{n\in\N\colon U\cap T^{-n}U\neq\emptyset\}
$$
for the return times of the set $U$ to itself.  
In case of ambiguity,  we include the transformation in our notation and write $N_T(x,U)$ or $N_T(U)$.
\end{notation}

Thus $R\subset \N$ is a set of recurrence if for every minimal system $(X,T)$ and every nonempty open set $U\subset  X$, there exists $n\in R$ such that $n\in N(U)$.

We recall a standard definition: 
\begin{definition}
A subset of integers is {\em syndetic} if the 
differerence between two consecutive elements is bounded.  
\end{definition}

We have the following classical equivalences 
(see, for example~\cite{FW, furstenberg, furstenberg1, bergelson0, Bergelson, McCutcheon2}).  
We omit the proofs, as simple recurrence is a special case of the more general result for multiple recurrence  (Theorem~\ref{th:mult-top-recur}):

\begin{theorem}
\label{th:rec}
For a set $R\subset \N$, the following are equivalent:
\begin{enumerate}
\item
\label{it:top-rec2}
$R$ is a set of recurrence.
\item
\label{it:top-rec1}
For every system $(X,T)$ and every open cover $\CU=(U_1,\dots,U_r)$ of $X$, there exists $j\in\{1,\dots,r\}$ and $n\in R$ such that 
$n\in N(U_j)$.
\item
\label{it:large}
For every finite partition $\N=C_1\cup\dots\cup C_r$ of $\N$,  there is some cell $C_j$ containing two integers whose difference belongs to $R$.
\item
\label{it:intersective}
Every syndetic subset $E$ of $\N$ contains two elements whose difference belongs to $R$.
\item
\label{it:returnx1}
For every system $(X,T)$ and every $\ve>0$, there exist  $x\in X$ and  $n\in R$ such that  $d(T^nx,x)<\ve$.
\item
\label{it:returnx}
For every system $(X,T)$, there exists $x\in X$ 
such that $$\inf_{n\in R} d(T^{n}x,x) = 0.$$
\item
\label{it:returnx2}
For every minimal system $(X,T)$ there exists a dense $G_\delta$ set $X_0\subset  X$ such that for every $x\in X_0$, 
$$\inf_{n\in R} d(T^{n}x,x) = 0.$$
\end{enumerate}
\end{theorem}

A set $R$ satisfying characterization~\eqref{it:intersective} is referred to as {\em (chromatically) intersective} in the combinatorics literature.

It is easy to check that the existence of some  $n\in R$ satisfying any of properties~\eqref{it:top-rec2}, \eqref{it:top-rec1} or~\eqref{it:returnx1}
implies that there exist infinitely many $n\in R$ with the same property. 

\begin{example}
\label{remark:diff}
For $S\subset\N$, write $S-S=\{s'-s\colon  s,s'\in S,\ s'>s\}$.
Furstenberg~\cite{furstenberg}
showed that if $S$ is infinite, then $S-S$ is a set of recurrence and 
this follows immediately from characterization~\eqref{it:large} in Theorem~\ref{th:rec}.
More generally, it is easy to check that if for every $n\in\N$ there exists $S_n\subset\N$ such that $|S_n|=n$ and $S_n-S_n\subset R$,  then $R$ is a set of recurrence. 
\end{example}
We defer further examples of sets of recurrence until we have defined the more general notion of multiple recurrence.

There is another equivalent formulation of recurrence due to Katznelson~\cite{katznelson}.  For a set $R\subset \N$, 
the {\em Cayley graph $G_R$} is defined to be the graph whose vertices are the natural numbers $\N$ and whose 
edges are the pairs $\{(m, m+n)\colon m\in\N, n\in R\}$.  The {\em chromatic number $\chi(R)$} is defined to be the smallest number of colors needed to color $G_R$ such that any two vertices connected by an edge have distinct colors.  
Katznelson showed that characterization~\eqref{it:large} of Theorem~\ref{th:rec} for a set of recurrence $R$ is equivalent to the associated Cayley graph $G_R$ having infinite chromatic number.

For the analogous notion of a set of measure theoretic recurrence, where the underlying space is a probability measure space and the transformation is a measurable, measure preserving transformation, we have a similar 
list of equivalent characterizations, where a finite partition of $\N$ is replaced by 
sets of positive upper density.  
As every minimal system $(X,T)$ admits a $T$-invariant measure with full support, a set of measurable recurrence is also a set of topological recurrence.  However, an intricate example of Kriz~\cite{Kriz} shows that the converse does not hold.  

\subsection{Multiple recurrence}
\label{sec:multiple}

Most of the formulations of single recurrence generalize to multiple recurrence: 
\begin{notation}
For $\ell\geq 1$, we write 

$$
N^\ell(U)=\{n\in\N\colon U\cap T^{-n}U\cap T^{-2n}U\cap\dots\cap
T^{-\ell n}U\neq\emptyset\}
$$
for the return times of the set $U$ to itself along a progression of length $\ell+1$.  
In case of ambiguity,  we include the transformation in our notation and write $N^\ell_T(U)$.
\end{notation}
\begin{theorem}
\label{th:mult-top-recur}
Let $\ell\geq 1$ be an integer. 
For a set $R\subset \N$, the following properties are equivalent:
\begin{enumerate}
\item
\label{it:top-rec2l}
For every minimal system $(X,T)$ and every nonempty open set $U\subset  X$, there exists $n\in R$ such that $n\in N^\ell(U)$.
\item
\label{it:top-rec1l}
For every system $(X,T)$ and every open cover $\CU=(U_1,\dots,U_r)$ of $X$, there exists $j\in\{1,\dots,r\}$ and $n\in R$ such that $n\in N^\ell(U_j)$.
\item
\label{it:largel}
For every finite partition $\N=C_1\cup\ldots\cup C_r$ of $\N$, there is some cell $C_j$ that contains an arithmetic progression of length $\ell+1$ whose common difference belongs to $R$.
\item
\label{it:intersectivel}
Every syndetic set $E\subset \N$ contains 
 an arithmetic progression of length $\ell+1$ whose common difference belongs to $R$.
  \item
\label{it:returnxl}
For every system $(X,T)$ and every $\ve>0$, there exist $x\in X$ and  $n\in R$ such that  
$$
\sup_{1\leq j\leq \ell} d(T^{jn}x,x) <\ve.$$
\item
\label{it:returnx3l}
For every system $(X,T)$, there exists $x\in X$ such that 
$$\inf_{n\in R}\sup_{1\leq j\leq \ell} d(T^{jn}x,x) = 0.$$
\item
\label{it:returnx2l}
For every minimal system $(X,T)$, there exists a dense $G_\delta$-set $X_0\subset  X$ such that for every $x\in X_0$, 
$$\inf_{n\in R}\sup_{1\leq j\leq \ell} d(T^{jn}x,x) = 0.$$
\end{enumerate}
\end{theorem}

\begin{definition}
A set satisfying any of the equivalent properties in Theorem~\ref{th:mult-top-recur}
is called a {\em set of $\ell$-recurrence}; in particular, a set of $1$-recurrence is a set of recurrence.  A set of  $\ell$-recurrence for all $\ell\geq 1$ is a called a {\em set of multiple  recurrence}.
\end{definition}

When we want to emphasize that we are discussing single recurrence, instead of just writing a set 
of recurrence, we say a set of {\em single} or {\em simple} recurrence. 

The proofs of these equivalences are well known and appear scattered in the literature (see, for example~\cite{FW, furstenberg, furstenberg1, bergelson0, Bergelson, McCutcheon1, McCutcheon2, nikos, BR, FM, BG}) and so  we only include brief sketches of the proofs.

\begin{proof}
\eqref{it:top-rec2l} $\Longrightarrow$ \eqref{it:returnx2l}
For $\ve>0$,  define $\Omega_\ve$ to be 
$$\{x\in X\colon\text{ there exists } n\in R \text{ such that } d(T^nx,x)<\ve,\dots d(T^{\ell n}x,x)<\ve\}.
$$
Then $\Omega_\ve$ is an open subset of $X$. Let $U\subset  X$ be an open ball of radius $\delta<\ve/2$.  By hypothesis, there exists $n\in R$ such that
$U\cap T^{-n}U\cap\dots\cap T^{-\ell n}U\neq\emptyset$. This intersection is included in $\Omega_\ve$ and so 
$\Omega_\ve$ is dense in $X$.   Then $X_0 = \bigcap_{m\in\N}\Omega_{1/m}$ is a $G_\delta$ set that  satisfies the statement.  

\eqref{it:returnx2l} $\Longrightarrow$ \eqref{it:returnx3l}
This is immediate by applying~\eqref{it:returnx2l} to a minimal closed invariant subset of $X$.

\eqref{it:returnx3l} $\Longrightarrow$  \eqref{it:returnxl} Obvious.

\eqref{it:returnxl} $\Longrightarrow$ \eqref{it:top-rec1l}
Let $\ve$ be the Lebesgue number of the cover $\CU$, meaning that any open ball of radius $\ve$ is  contained in some element of this cover. 
Let $x\in X$ and $n\in R$ be associated to $\ve$ as in \eqref{it:returnxl}. Let  $j\in\{1,\dots,r\}$ be such that the ball of radius $\ve$ around $x$ is included in $U_j$.  Then all of the points $x,T^nx,\dots,T^{\ell n}x$ belong to this ball and thus to $U_j$.

\eqref{it:top-rec1l} $\Longrightarrow$ \eqref{it:largel}. This is a standard application of the topological version of Furstenberg's Correspondence Principle.  Given the partition $\N=C_1\cup\dots\cup C_r$, there exist a system $(X,T)$, a partition $X=U_1\cup\dots\cup U_r$ of $X$ into clopen sets, and a point $x\in X$ such that for every $n\in\N$, we have $T^nx\in U_j$ if and only if $n\in C_j$.

\eqref{it:largel} $\Longrightarrow$ \eqref{it:intersectivel}
Choose $r\in\N$ such that $(E-1)\cup(E-2)\cup\ldots \cup(E-r)\supset \N$ and then chose a partition $\N=C_1\cup\dots\cup C_r$ such that $C_j\subset  E-j$ for $j\in\{1, \ldots, r\}$.

\eqref{it:intersectivel} $\Longrightarrow$  \eqref{it:top-rec2l}
Choose $x\in X$ and set $E=\{n\colon T^nx\in U\}$.
\end{proof}

As for single recurrence, the existence of some $n\in R$ satisfying any of properties~\eqref{it:top-rec2l}, 
\eqref{it:top-rec1l}, or \eqref{it:returnxl} immediately implies the existence of infinitely many $n\in R$ with the same property. 

It is easy to verify that a set of (single or multiple) recurrence must satisfy several necessary conditions: it must contain infinitely many multiples of every positive integer (consider the powers of the transformation) and it can not be lacunary (by constructing an irrational rotation that fails to recur).  Furthermore, the family of sets of recurrence has the 
Ramsey property (see Section~\ref{sec:Ramsey}).

The classic theorem of van der Waerden shows that $\N$ 
is a set of multiple recurrence.  Furstenberg~\cite[Theorem 2.16]{furstenberg} shows that $N(x, U)$ is a set of multiple recurrence for any open set $U$ and point $x\in U$.  
This is also a particular case of a more general theorem of Huang, Song, and Ye~\cite{HSY}, reviewed in Theorem~\ref{theorem:Nxu-rec}. 

There are many other known examples of sets of multiple recurrence: any $\ip$-set  (a set which contains all the finite sums of an infinite set of integers, see Definition~\ref{def:IP}),  the set $\{p(n)\colon n\in\N\}$, where $p(n)$ is any non-constant polynomial with $p(0) = 0$,  the shifted primes $\{p-1\colon p\text{ is prime}\}$ and $\{p+1\colon p\text{ is prime}\}$, 
as well as other examples in the literature (see for example~\cite{FW, wierdl, BL, Bergelson, BHM, FHK})

There are also examples in the literature that show that sets of multiple recurrence are different than sets of single recurrence.  For example, Furstenberg~\cite{furstenberg} gives an example of a set of single recurrence that is not a set of double recurrence and Frantzikinakis, Lesigne and Wierdl~\cite{FLW} give examples of  sets of $\ell$-recurrence that are not sets of $(\ell+1)$-recurrence.  We give a more general characterization of such sets in Section~\ref{sec:mult-RP}. We note that all of the examples constructed in this way are large, in the sense that they have positive density.

However, there are characterizations of single recurrence for which we do not have a multiple analog: 
\begin{question}
\label{question:multiple1}
Is there an equivalent characterization of multiple recurrence analogous to Katznelson's characterization in terms of the chromatic number of an associated graph?  For example, is being a set of multiple recurrence equivalent to infinite chromatic number for some  associated hypergraph? 
\end{question}

Along similar lines, we do not know of a simple construction, like that of the difference set, that suffices to produce multiple recurrence: 
\begin{question}
\label{question:multiple2}
Is there a  sufficient condition, analogous to that given in Example~\ref{remark:diff}, 
 that suffices for being a set of multiple recurrence?
\end{question}

\subsection{Simultaneous recurrence} 

More generally, we can study recurrence for commuting transformations and not just powers of a single transformation: 
\begin{definition}
The set $R\subset \N$ is a set of {\em $\ell$-simultaneous recurrence} if for any compact metric space $X$ endowed with 
$\ell$ commuting homeomorphisms $T_1, \ldots, T_\ell\colon X\to X$ 
such that the system $(X, T_1, \ldots, T_\ell)$ is minimal and any nonempty open set $U\subset  X$, there exists $n\in R$ such that 
$$U\cap T^{-n}_1U\cap\ldots\cap T^{-n}_\ell U\neq\emptyset.
$$

 A set of  $\ell$-simultaneous recurrence for all $\ell\geq 1$ is a called a {\em set of simultaneous recurrence}.
\end{definition}

Taking $T_1 = T, T_2 = T^2, \ldots, T_\ell = T^\ell$, 
it is obvious that any set of simultaneous recurrence is also a set of multiple recurrence.  We do not know 
if the converse holds: 
\begin{question}
Does there exist a set of multiple recurrence that is not a set of simultaneous recurrence? 
\end{question}

All of the examples of sets of multiple recurrence given in Section~\ref{sec:multiple} are also known to be 
sets of simultaneous recurrence.

All parts of Theorem~\ref{th:mult-top-recur} have natural analogs for simultaneous recurrence.  
To ease the notations, we restrict ourselves to $\ell=2$.  
It is easy to check that the analog of condition~\eqref{it:largel} holds: 
namely, $R$ is a set of recurrence  if for every partition $\N = C_1\cup\ldots\cup C_r$, 
there exists $x,y\in\N$ and $n\in R$ such that 
$(x,y), (x+n, y), (x, y+n)$ all lie in the same cell $C_j$ for some $j\in\{1, \ldots, r\}$.  
One can give similar formulations for the other equivalences in Theorem~\ref{th:mult-top-recur} for simultaneous 
recurrence.  

Unsurprisingly, we do not know how to address the analogs of Questions~\ref{question:multiple1} and~\ref{question:multiple2} for simultaneous recurrence.

\subsection{Pointwise recurrence}

\begin{definition}
A set $R\subset \N$ is \emph{a set of pointwise recurrence} if for every minimal system and  every $x\in X$, 
$$
\inf_{n\in R}d(T^nx,x)=0.
$$\end{definition}

The analog for multiple pointwise recurrence is not defined, as one can construct an example (such as using 
symbolic dynamics) of a minimal system $(X,T)$, as open set $U\subset X$, and $x\in U$ such that $N^2(x,U) = \emptyset$.  In 
particular, $\N$ is not a set of pointwise multiple recurrence.  However, in a minimal system, there is always 
a dense set of points that are multiply recurrent.

Recall that by characterization~\eqref{it:returnx2} of Theorem~\ref{th:rec}, if $R$ is a set of  recurrence then this property holds for  $x$ in a dense $G_\delta$ of $X$.
Comparing the definition of pointwise recurrence 
with characterization~\eqref{it:returnx} of  recurrence in Theorem~\ref{th:rec} makes this property seem natural.  
However, being a set of pointwise recurrence turns out to be a significantly stronger assumption.
S\'ark\H ozy~\cite{sarkozy} (using number theoretic methods) and Furstenberg~\cite{furstenberg} (using dynamics) showed that the 
set of squares is a set of recurrence, but Pavlov~\cite{pavlov} showed that it is not a set of pointwise recurrence.  
Similarly, it follows from results in Pavlov that if one takes $S$ to be a sufficiently fast growing sequence, then 
$S-S$ is not a set of pointwise recurrence (but as noted in Example~\ref{remark:diff}, it is a set of recurrence).  

\begin{notation}
For $t\in\R$, we use $\norm{t}$ to denote the distance of $t$ to the nearest integer.  For $t\in\T=\R/\Z$, $\norm{t}$ denotes the distance to $0$.
\end{notation}

\begin{example}
One can check directly that  for every $\alpha\in\T$, the set $R = \{n\in \N\colon \norm{n^2\alpha}\geq 1/4\}$ is not a set of pointwise recurrence by using an affine nilsystem (see Example~\ref{example:FLW}).  
In~\cite{FLW}, the authors show, in particular,  that $R$ is a set of measurable recurrence, 
and thus also of recurrence. We briefly outline their method. If $\alpha$ is irrational, by Weyl equidistribution,  for every non-zero $t\in[0,1)$, the averages 
$$\frac{1}{N}\sum_{n=1}^N e^{2\pi i kn^2\alpha} e^{2\pi i nt} $$
converge to $0$ as $N\to\infty$ for every non-zero integer $k$.  It follows that  the averages
$$
\frac{1}{N}\sum_{n=1}^N\one_I(n^2\alpha) e^{2\pi i nt},
$$ 
where $I = [1/4, 3/4]$, converge to $0$ for $t\neq 0$ when $N\to\infty$ and that the limit is $1/2$ when $t=0$. 
By the spectral theorem, it follows that 
for any ergodic measure preserving system $(X, \CB, \mu, T)$ and $A\in \CB$ with $\mu(A)>0$, 
$$
\frac{1}{N}\sum_{n=1}^N \one_I(n^2\alpha)\mu(A\cap  T^n A) \to \frac{1}{2}\mu(A)^2, 
$$
and the positivity of the limit implies the recurrence.  

A generalization of this example is given in Corollary~\ref{cor:Nxrecir}. 
\end{example}

We ask if there exist equivalent characterizations of pointwise recurrence: 
\begin{question}
Is there a combinatorial analog of pointwise recurrence?  Are there sufficient conditions for being a set of 
pointwise recurrence? 
\end{question}

While simple recurrence does not imply multiple recurrence (see further discussion in Section~\ref{sec:mult-RP}), 
this may hold under the stronger notion of pointwise recurrence: 
\begin{question}
Does pointwise recurrence imply multiple recurrence?
\end{question}

We give a partial answer to this question in Section~\ref{sec:distal}.
\section{Recurrence for families of systems} 
\label{sec:recurrence-for-families}

\subsection{Questions for families of systems}

We define the notion of a set of $\ell$-recurrence for a given system in the obvious way:
\begin{definition}
If $\CF$ is a family of systems, a set $R\subset \N$ is a {\em set of  recurrence for the family $\CF$} if for any minimal
system $(X,T)$ in the family  $\CF$ and any nonempty open set $U\subset  X$, there exists $n\in R$ such that 
$U\cap T^{-n}U\neq\emptyset$.  The notions of a \emph{set of $\ell$-recurrence} and a \emph{set of multiple recurrence for the family $\CF$} are defined in the same way.
\end{definition}

We can take the family $\CF$ to be rotations, nilsystems, distal systems, or any other class of systems.  
While it is obvious that a set of recurrence in some class is a set of recurrence for a sub-class, we are interested in the converse.  Broadly stated, we 
ask: for which classes of systems does recurrence or multiple recurrence imply the same property in some larger class?

Furthermore, we are interested in relations between the various notions of recurrence.  We have different types of recurrence, including single, multiple, and pointwise recurrence, all of which are distinct notions.  For which classes of systems do these properties coincide?

We study these questions for distal systems in 
 Section~\ref{sec:distal} and for 
 nilsystems in Sections~\ref{sec:lift-Bohr} and~\ref{sec:mmultiple}.

While the equivalent formulations of multiple recurrence that are dynamical in nature carry over for the restriction to particular families of systems, we do not have combinatorial equivalences for classes of systems, 
and it is natural to ask if there are combinatorial versions of recurrence for particular classes of systems.

\subsection{Bohr recurrence}
\label{sec:Bohr-charac}
We start with the simplest types of systems:
\begin{definition}
A set of recurrence for minimal translations on a compact abelian group is called a set of {\em Bohr recurrence}. 
\end{definition}
Thus $R$ is a set of Bohr recurrence if for all $k\in\N$, all $\alpha_1, \ldots, \alpha_k\in\T$, and all $\varepsilon > 0$, there exists $n\in R$ such that 
 $ \norm{\alpha_1n}<\ve,\dots,\norm{\alpha_kn}<\ve$.
 It follows immediately that there are infinitely many $n\in R$ satisfying this condition.  
 
We can also define a set of Bohr recurrence in terms of \bohrzero sets:
\begin{definition}
 A set $E\subset \N$ is a \emph{\bohrzero set} if it
  contains a set of integers of the form
$$
\{n\in\N\colon \norm{\alpha_1n}<\ve,\dots,\norm{\alpha_kn}<\ve\}, 
$$
where $k\in\N$, $\alpha_1,\dots,\alpha_k\in\T$ and $\ve>0$. The minimum value of $k$ such that this occurs is called the \emph{dimension} of the \bohrzero set.
\end{definition}

It follows immediately from the definitions that a set is a set of Bohr recurrence if and only if 
it is a \bohrzeroo set, meaning it has nonempty intersection with any \bohrzero set.  
Thus a set of Bohr recurrence is a set of multiple pointwise recurrence for translations on a compact abelian group, with no assumption of minimality required.

A well known question, asked in particular by Katznelson (see also the discussion in~\cite{weiss}) is:
\begin{question}[Katznelson~\cite{katznelson}]
\label{q:katznelson}
Is a set of Bohr recurrence a set of recurrence? 
\end{question}

This question leads us to multiple sub-questions about what types of 
extensions preserve sets of recurrence and of multiple recurrence.

\subsection{Recurrence and proximal extensions}

We start with the classic notion of a proximal extension (see, for example,~\cite{Aus}):
\begin{definition}
Let $(X,T)$ be a system. The points  $x_1,x_2\in X$ are \emph{proximal} if
$$
\inf_{n\in\N}d(T^nx_1,T^nx_2)=0.
$$
A set $F\subset  X$ is {\em proximal} if every pair of points in $F$ is proximal.

We say that the factor map $\pi\colon (X,T)\to (Y,S)$ is  a \emph{proximal extension} if the fiber $\pi\inv(\{y_0\})$ of every $y_0\in Y$ is proximal.
\end{definition}
In fact, this property holds under weaker assumptions: 
\begin{claim}
\label{cl:prox}
Let  $\pi\colon(X,T)\to (Y,S)$ be a factor map and assume that $(Y,S)$ is minimal and that some $y_0\in Y$ has a proximal fiber. Then $\pi$ is a proximal extension.
\end{claim}
\begin{proof}
Assume that the fiber of $y_0\in Y$ is proximal. 
For $x,x'\in X$, let $\delta(x,x')=\inf_{n\in\N} d(T^nx,T^nx')$ and for
 $y\in Y$, let 
 $$\phi(y)=\sup_{x,x'\in\pi\inv(\{y\})}\delta(x,x').$$ 
 Then $\delta$ is an upper semicontinuous function on $X\times X$ and satisfies 
 \break $\delta(Tx,Tx')\geq\delta(x,x')$ for all $x,x'\in X$.  Thus the function $\phi$ on $Y$ is upper semicontinuous and satisfies $\phi(Sy)\geq\phi(y)$. Since $\phi(y_0)=0$, we have that $\phi(S^{-n}y_0)=0$ for every $n\in\N$.  By minimality of $(Y,S)$, 
 we have that $\phi(y)=0$ for every $y\in Y$.
\end{proof}

Properties similar to the following lemma appear in different places in the literature. We provide a proof for completeness. 

\begin{lemma} 
\label{lem:proxi}
Let $\pi\colon (X,T)\to (Y,S)$ be a proximal extension between minimal systems. Then 
for every  $\ell\geq 1$ and all $x_0,\dots,x_\ell$ lying in the same fiber, there exists a sequence of integers $(n_i)$ such that each of the sequences $(T^{n_i}x_0)$, \dots, $(T^{n_i}x_\ell)$ converge to $x_0$.
\end{lemma}

\begin{proof}
We proceed by induction on $\ell$. Assume that $\ell=1$, and let $y_0\in Y$, $x_0,x_1\in\pi\inv(\{y_0\})$, and $\ve>0$. By proximality, there exists a sequence of integers $(n_i)$ such that the sequences $(T^{n_i}x_0)$ and $(T^{n_i}x_1)$ converge to the same point $a\in X$. By minimality of $(X,T)$, there exists $m\in\N$ with $d(T^ma,x_0)<\ve/2$. By continuity of $T^m$, for every sufficiently large $i$ and $k=0,1$, we have $d(T^{n_i+m}x_k,x_0)<\ve$.  The result follows for $\ell=1$.

Assume that $\ell>1$ and that the result holds with $\ell-1$ substituted for $\ell$. Let $y_0\in Y$, $x_0,\dots,x_\ell\in\pi\inv(\{y_0\})$, and $\ve>0$. By the induction hypothesis, there exists a sequence of integers $(n_i)$ such that the sequences $(T^{n_i}x_k)$, $0\leq k\leq \ell-1$, converge to $x_0$. Passing to a subsequence, we can assume that the sequence $(T^{n_i}x_\ell)$ converge to a point $a\in X$. For every $i$, we have $\pi(T^{n_i}x_\ell)=\pi(T^{n_i}x_0)$ and, passing to the limit, $\pi(a)=\pi(x_0)=y_0$.
By applying the result for $\ell=1$ to the points $x_0$ and $a$, we obtain the existence of $m\in \N$ with $d(T^mx_0,x_0)<\ve/2$ and $d(T^ma,x_0)<\ve/2$.
By continuity of $T^m$, for every sufficiently large $i$ and every $k$ with $0\leq k\leq\ell$, we have 
$d(T^{n_i+m}x_k,x_0)<\ve$, completing the proof.
\end{proof}

\begin{proposition}
\label{prop:proximal}
Let $\pi\colon (X,T)\to (Y,S)$ be a proximal extension between minimal systems, $\ell\geq 1$, and $R$ be a set of $\ell$-recurrence for $(Y,S)$. Then $R$ is a set of $\ell$-recurrence for $(X,T)$.
\end{proposition}
In particular, this proposition applies to almost 1-1 extensions and asymptotic extensions between minimal systems.  For example, a set of Bohr 
recurrence is a set of multiple recurrence for Sturmian systems, as Sturmian 
systems are almost 1-1 extensions of rotations.  
\begin{proof}
Let $\ve >0$.
By characterization~\eqref{it:returnx3l} in Theorem~\ref{th:mult-top-recur} of sets of $\ell$-recurrence,  there exists $y_0\in Y$ such that 
$$
\inf_{n\in R}\sup_{1\leq k\leq\ell}d(S^{kn}y_0,y_0)=0
$$
and thus there exists a sequence $(n_i)$ in $R$ such that $
S^{kn_i}y_0\to y_0$ for $1\leq k\leq \ell$.

Let  $x_0\in X$ with $\pi(x_0)=y_0$.  
Passing to a subsequence, we can assume that 
$$
\text{for }1\leq k\leq\ell,\text{ the sequence }(T^{kn_i}x_0)\text{ converges in }X.
$$
Letting $x_k$ denote the limit of this sequence, we have that $\pi(x_k)=y_0$

The points $x_0,x_1,\dots,x_\ell$ belong to the fiber $\pi\inv(\{y_0\})$ and this fiber is proximal by hypothesis. By Lemma~\ref{lem:proxi}, there exists a sequence of integers $(m_j)$ such that the sequences $(T^{m_j}x_k)$, $0\leq k\leq\ell$, converge to $x_0$.

Choose $j$ such that 
$$
d(T^{m_j}x_k,x_0)<\ve\text{ for }0\leq k\leq\ell.
$$  
Let $\delta>0$ be such that for $x,x'\in X$,  $d(x,x')<\delta$ implies $d(T^{m_j}x,T^{m_j}x')<\ve$ and let
$i$ be such that $d(T^{kn_i}x_0,x_k)<\delta$ for $1\leq k\leq\ell$. We have that
$d(T^{m_j+kn_i}x_0, T^{m_j}x_k)<\ve$ and $d(T^{m_j+kn_i}x_0, x_0)<2\ve$.

Letting $z=T^{m_j}x_0$, we have that $d(x_0,z)<\ve$ and $d(T^{kn_i}z,x_0)<2\ve$ for $1\leq k\leq\ell$.
By characterization~\eqref{it:returnxl} in Theorem~\ref{th:mult-top-recur} restricted to such systems, $R$ is a set of $\ell$-recurrence for $X$.
\end{proof}

Recall that $\pi\colon (X,T)\to (Y,S)$ is a {\em distal extension} if 
all $x_0\neq x_1\in X$ in the same fiber satisfy 
$\inf_n d(T^nx_0, T^nx_1)> 0$.  
Proposition~\ref{prop:proximal} does not generalize to distal extensions (see the remarks after Corollary~\ref{cor:Nxrecir}), even for simple recurrence, as can be seen by taking an extension of the trivial system.  
However, for the class of nilsystems, this is possible (Theorem~\ref{th:rec-nil}). 

\subsection{Pointwise recurrence in a distal system}
\label{sec:distal}

Recurrence forces structure in return times and this is captured in the notion of $\ip$-sets (see~\cite{furstenberg, Bergelson} 
for background): 
\begin{definition}\label{def:IP}
An {\em $\ip$-set} is a set of integers that contains an infinite sequence of integers $(p_i)_{i\in\N}$, the {\em generators},  and 
all the finite sums $\sum_{j=1}^k p_{i_j}$, where the summands are distinct generators and $k=1, 2, \ldots$.  
An {\em $\ip^*$-set} is a set of integers that has nontrivial intersection with every $\ip$-set. 
A point $x\in X$ is said to be {\em $\ip^*$-recurrent} if for every open neighborhood $U$ of $x$, 
$\{n\in\N\colon T^nx\in U\}$ is an $\ip^*$-set.
\end{definition}

It is easy to check that the return times of any recurrent point contains an $\ip$-set and conversely that for any $\ip$-set, 
there is a dynamical system and a recurrent point whose return times contain this $\ip$-set.  Furthermore, 
Furstenberg~\cite{furstenberg} shows that pointwise recurrence and $\ip$-sets are closely related: for a 
distal system, every point is $\ip^*$-recurrent.  Thus: 
\begin{proposition}
Every $\ip$-set is a set of pointwise recurrence for distal systems.
\end{proposition}
\begin{question}
Is it true that every set of pointwise recurrence for distal systems is an $\ip$-set? 
\end{question}
We believe that there should be a counterexample.  

In a distal system, pointwise recurrence implies the multiple version: 
\begin{proposition}
\label{prop:mult-strong-distal}
A set of pointwise recurrence for distal systems is a set of pointwise multiple recurrence for distal systems.
\end{proposition}
\begin{proof}
Let $\ell\geq 1$, $\wt T$ denote the transformation $T\times T^2\times\dots\times T^\ell$ of $X^\ell$, and $\wt X$ the closed orbit of the point $\wt x=(x,x,\dots,x)\in X^\ell$ under $\wt T^\ell$. Then $(\wt X,\wt T)$ is transitive and distal, and so it is minimal.  

Since $R$ is a set of pointwise recurrence, for every $\ve >0$ there exists $n\in R$ with
$\wt d(\wt T^n\wt x,\wt x)<\ve$, that is, $d(T^nx,x)<\ve$, \dots,
$d(T^{\ell n}x,x)<\ve$.
\end{proof}

More generally, the same result holds for an {\em almost distal system}, meaning a system in which every pair of points is either asymptotic (in this context, one sided asymptotic) or distal (see~\cite[Theorem 3.10]{BGKM} for more on almost distal systems).

It is easy to check that if $R$ is a set of pointwise recurrence for a distal minimal system, 
then we have a seemingly stronger property.
Let $(X_1, T_1), \ldots,$ $(X_k, T_k)$ be distal systems 
and for $1\leq j \leq k$, let $U_j$ be a nonempty, open 
subset of $X_j$.  Then there exist $n\in R$ such that 
$T_j^{-n}U_j\cap U_j\neq\emptyset$ for $j=1, \ldots, k$.  To see this, choose $x_j\in U_j$ for each $j$ and define $x = (x_1, \ldots, x_k)\in X_1\times\ldots\times X_k$ and let $T = T_1\times\ldots\times T_k$.  Proceeding as in the proof of Proposition~\ref{prop:mult-strong-distal}, we have the statement.

\section{Recurrence in nilsystems}
\label{sec:lift-Bohr}
\subsection{Nilsystems}

Let $G$ be a nilpotent Lie group. The {\em commutator} of $a,b\in G$ is defined to be $[a,b]= aba^{-1}b^{-1}$ and for $A, B\subset  G$, we let $[A,B]$ denote the group spanned by $\{[a,b]\colon a\in A, b\in B\}$. The {\em commutator subgroups} $G_j$ of $G$ are defined inductively, with $G_1 = G$ and for integers $j \geq 1$, we have $G_{j+1} = [G, G_j]$. For an integer $s\geq 1$,  if $G_{s+1} = \{1_G\}$ then 
$G$ is said to be {\em $s$-step nilpotent}.  

Let $s\geq 1$ be an integer, $G$ be an $s$-step nilpotent Lie group, and $\Gamma$ be a discrete cocompact subgroup of $G$. Then the compact nilmanifold $X=G/\Gamma$ is an {\em $s$-step nilmanifold}. Viewing elements of $X$ as points rather than congruence classes, we write $e_X$ for the image of the identity $1_G$ in $X$ and write $(g,x)\mapsto g\cdot x$ for the natural action of $G$ on $X$. 
Let $T\colon X\to X$ be the transformation $x\mapsto\tau\cdot x$ for some fixed element $\tau\in G$. Then $(X,T)$ is an {\em $s$-step nilsystem}.
Thus a $1$-step nilsystem is exactly  a translation  on a compact abelian group.

We note that  we do not assume that $G$ is connected, as this excludes some interesting examples, such as the \emph{affine nilsystems} defined in Section~\ref{sec:affine}.

In the next theorem, which is the main result of this section, we answer Question~\ref{q:katznelson} for the class of nilsystems:
\begin{theorem}
\label{th:rec-nil}
Let $R\subset \N$ be a set of Bohr recurrence. Then for every integer $s\geq 1$, $R$ is a set of recurrence for minimal $s$-step nilsystems.
\end{theorem}
It immediately follows that the result also holds for inverse limits of nilsystems, 
and it follows for proximal extensions of these systems by Proposition~\ref{prop:proximal}.  Particular examples are Sturmian or Toeplitz systems, but also more complicated constructions such as almost one to one extensions of infinite-step nilsystems (see \cite{DDMSY}). 

The rest of this section is devoted to two proofs of Theorem~\ref{th:rec-nil}. The first one uses measure theoretic arguments; it is shorter than the second one but unfortunately requires an additional hypothesis. The second proof is more technical and is completely topological, and has the possible advantage that 
it may be generalized.

We start by recalling some properties of nilsystems, referring to~\cite{AGH, parry, leibman} for background 
and further details.   Let  $(X=G/\Gamma,T)$ be an $s$-step nilsystem. 
Henceforth we assume that $(X,T)$ is minimal.  

Let $d_G$ denote a right invariant distance on the group $G$ that defines its topology, and assume that $X$ is endowed with the quotient distance, which we denote as  $d_X$.

For $j=1, \ldots, s$, we have that $G_j$ and $G_j\Gamma$ are closed subgroups of $G$.  
Let $G_0$ denote the connected component of $1_G$ in $G$. Then $G_0$ is an open, normal subgroup of $G$.  By the assumption of minimality, we can assume that $G=\langle G_0,\tau\rangle$ and we make this assumption in the sequel.  
This in turn implies that the commutator group $G_2$ is connected and included in $G_0$.

Set $Z=G/G_2\Gamma$. Then $Z$ is a compact abelian group. The natural projection  $X\to Z$ is a factor map of $(X,T)$ to $Z$, endowed with the translation by the image $\alpha$ of $\tau$ in $Z$, and 
the system $Z$ endowed with this translation is minimal.

\subsection{Measure theoretic proof of Theorem~\ref{th:rec-nil} under an additional assumption}

Maintaining the same notation, we continue to assume that the  $s$-step nilsystem $(X,T)$ is minimal.  Thus it is uniquely ergodic and its invariant measure is the \emph{Haar measure} $\mu$ of $X$. Let $m_Z$ denote the Haar measure of  $Z=G/G_2\Gamma$. Recall that $\alpha$ is the image of $\tau$ in $Z$ and let  $S$ be the translation by $\alpha$ on $Z$. Then $\pi\colon(X,\mu,T)\to(Z,m_Z,S)$ is a measure theoretic factor map, and more precisely $(Z,m_Z,S)$ is the Kronecker factor of $(X,\mu,T)$. 
We use additive notation in $Z$, and assume that this abelian group is endowed with a translation invariant distance $d_Z$ defining its topology.

The following result is proven for connected $G$ in~\cite{AGH},  using the theory of representations of nilpotent Lie groups. An elementary proof for the case of $2$-step nilsystems is given in~\cite{HKM1}.

\begin{theorem*}[\cite{AGH}, \cite{HKM1}]
If $G$ is connected or if $s=2$, then the spectral measure of any function $f\in L^2(\mu)$ with $\E(f\mid Z)=0$ is absolutely continuous.  
\end{theorem*}

This means that if $\E(f\mid Z)=0$, then the finite measure $\sigma_f$ on $\T$ defined by
$$
\wh\sigma_f(n)=\int T^nf\cdot \overline f\,d\mu\text{ for }n\in\Z
$$ is absolutely continuous with respect to Lebesgue measure on $\T$. This implies in particular that $\wh\sigma_f(n)\to 0$ when $n\to+\infty$.

\begin{proposition}
\label{prop:spectral}
Let $(X=G/\Gamma,T)$ be a minimal $s$-step nilsystem.
The statement of Theorem~\ref{th:rec-nil} holds if $G$ is connected, as well as for  $s=2$ without any further assumptions.
\end{proposition}
\begin{proof}
Let $R\subset \N$ be a set of Bohr recurrence and let $U$ be a nonempty, open subset of $X$. We want to show that there exists $n\in R$ with $U\cap T^{-n}U\neq \emptyset$. It suffices to show that this set has positive measure. 

By minimality of $(X,T)$, the Haar measure $\mu$ has full topological support and thus $\mu(U)>0$.
Set $\ve=\mu(U)^2/4$.
Define 
$$
g=\E(\one_U|Z)\text{ and }f=\one_U-g
$$
and note  that $1\geq \norm g_{L^2(m_Z)}\geq\norm g_{L^1(m_Z)}=\mu(U)$.
We have $\E(f\mid Z)=0$ and thus there exists $n_0\in\N$ such that 
$$
\Bigl|\int T^nf \cdot f\,d\mu\Bigr|<\ve\text{ for every }n\geq n_0.
$$
 On the other hand, writing $g_t(z)=g(z+t)$ for $t\in Z$, there exists $\delta>0$ such that $\norm{g_t-g}_{L^2(m_Z)}<\ve$ for every $t \in Z$ with 
 $d_Z(t,0)<\delta$. 
 In particular, 
 $$
 \norm{S^ng-g}_{L^2(m_Z)}<\ve\text{ for every $n$ such that }d_Z(n\alpha,0)<\delta.
 $$ 
 Since $R$ is a set of Bohr recurrence, there exists $n\in R$ with $n\geq n_0$ and   $d_Z(n\alpha,0)<\delta$. For this value of $n$, since $\E(T^n\one_U|Z)=S^ng$, we have
\begin{align*}
\mu(U\cap T^{-n}U)
= & \int f\cdot T^nf\,d\mu+\int g\cdot S^ng\,dm_Z\\
\geq & -\ve+ \norm g_{L^2(m_Z)}^2-\norm g_{L^2(m_Z)}\norm{S^ng-g}_{L^2(m_Z)}\\
\geq & \mu(U)^2-2\ve\geq\mu(U)^2/2>0.\qed
\end{align*}
\renewcommand{\qed}{}
\end{proof}

\subsection{Topological proof of Theorem~\ref{th:rec-nil}}
We start by recalling some further facts about nilsystems (again, \cite{AGH, parry, leibman} are sources for background).

Let $(X,T)$ be a minimal $s$-step nilsystem, where T is the translation by $\tau\in G$. Recall that we assume that $G=\langle G_0,\tau\rangle$.
 If needed, we can represent this system  as a quotient $G/\Gamma$ where $G_0$ is simply connected and thus we can  assume this property without loss of generality. Set $\Gamma_0=\Gamma \cap G_0$. In this case, $G_0$ can be endowed with a \emph{Mal'cev basis}. Using this basis, we can identify $G_0/G_2$ with $\R^p$ for some integer $p\in\N$, such that the subgroup $\Gamma_0/(\Gamma_0\cap G_2)$ corresponds to $\Z^p$, and thus  $G_0/G_2\Gamma_0$ is identified with $\T^p$. Furthermore, the abelian group $G_s$ can be identified with  $\R^r$ for some $r\in\N$, 
 such that $\Gamma\cap G_s$ corresponds to $\Z^r$, inducing the identification of $G_s/(\Gamma\cap G_s)$ and $\T^r$. 
 Finally, $G_{s-1}/G_s$ is an abelian group, and is nontrivial  if $X$ is not an $(s-2)$-step nilsystem. In this case, this group can be identified   with $\R^q$ for some $q\in\N$, and such that the subgroup $(\Gamma\cap G_{s-1})/(\Gamma\cap G_{s})$ corresponds to $\Z^q$.

Moreover, the distance $d_G$ on $G$ can be chosen such  that these identifications are isometries when the quotient groups are endowed with the quotient distances and
$\R^p$, $\T^p$, $\R^r$, $\T^r$ and  $\R^q$ are endowed with the Euclidean distances.
We caution the reader that under this identification, groups such as $\T^p$ and $\R^r$ are written with 
additive notation, while groups such as $G_0/G_j\Gamma$ and $G_s$ are written with multiplicative notation.  

Assume now  that $s\geq 2$. Define
\begin{equation}
\label{eq:Xtilde}
\wt G:=G/G_s, \ \wt\Gamma:=\Gamma/(\Gamma\cap G_s), \ \text{ and } \wt X:= \wt G/\wt\Gamma.
\end{equation}
Then 
$\wt G$ is an 
$(s-1)$-step nilpotent group, $\wt\Gamma$ is a discrete cocompact subgroup, $\wt X$ is an 
$(s-1)$-step nilmanifold, and the quotient map $G\to\wt G$ induces a projection $\pi\colon X\to \wt X$.  Thus
 we can view $\wt X$ as the quotient of $X$ under the action of $G_s$. Let $\wt\tau$ be the image of $\tau$ in $\wt G$ and $\wt T$ be the translation by $\wt\tau$ on $\wt X$. Then $(\wt X,\wt\tau)$ is an $(s-1)$-step nilsystem and $\pi\colon X\to\wt X$ is a factor map.

Maintaining this notation: 
\begin{lemma}
\label{lem:densecommut}
Let $(X,T)$ be a minimal $s$-step nilsystem and assume that $X$ is connected and that $G_0$ is simply connected. Then for every $\ve>0$, there exists $C:=C(\ve)$ such that for every $w\in G_s$, there exist $h\in G_{s-1}$ and 
$\gamma\in \Gamma\cap G_s$ with
$$
d_G(h,1_G)<C\ ; \ d_G([h,\tau],w\gamma)<\ve.
$$
\end{lemma}
\begin{proof}
Since $X$ is connected, it follows that $G=\langle G_0,\Gamma\rangle$ and there exists $\tau_0\in G_0$ and $\gamma_0\in\Gamma$ such that $\tau=\tau_0\gamma_0$. (If $G$ is connected, we have $\tau_0=\tau$ and $\gamma_0=1_G$.)  

Recall that $\Gamma_0=\Gamma\cap G_0$. Since  $G=\langle G_0,\tau\rangle$, we have that $G=\langle G_0,\gamma_0\rangle$ and thus $\Gamma=\langle\Gamma_0,\gamma_0\rangle$. 

Recall also that  $Z:=G/(G_2\Gamma)=G_0/(G_2\Gamma_0)=\T^p$, and that 
 the image $\alpha$ of $\tau$ in $G/(G_2\Gamma)$ is an ergodic element.
Let $\beta$ be the projection of $\tau_0$ to $G_0/G_2=\R^p$. Then the projection of $\beta$ in $G_0/(G_2\Gamma_0)$ is equal to the projection $\alpha$ of $\tau$ in $G/(G_2\Gamma)$. It follows that the coordinates $(\beta_1,\dots,\beta_p)$ of $\beta$ are rationally independent.

Let $\pi_s\colon G_s\to G_s/(\Gamma\cap G_s)$ be the quotient map. We claim: 
\begin{claim*}
The map
$f\colon h\mapsto\pi_s([h,\tau])$ takes  $ G_{s-1}$ to a dense subset of $G_s/(\Gamma\cap G_s)$.
\end{claim*}
Assuming the claim, there exists $C>0$ such that the image under $f$ of the ball $B_G(1_G,C)\cap G_{s-1}$ is $\ve$-dense in $G_s/(\Gamma\cap G_s)$,  and this is the statement of the lemma.

To prove the claim, note that 
the map $g\mapsto [g,\gamma_0]$ induces a group homomorphism $F\colon G_{s-1}/G_s\to G_s$. Using additive notation and writing in coordinates,
$$
\text{for }1\leq i\leq r,\quad \bigl(F(x)\bigr)_i=\sum_{j=1}^q F_{i,j}x_j
$$
and, since $[G_{s-1}\cap\Gamma,\gamma_0]\subset G_s\cap \Gamma$, $F$ maps $(G_{s-1}\cap\Gamma)/(G_s\cap\Gamma)$ to $G_s \cap \Gamma$, we have that the coefficients $F_{i,j}$ are integers.

The commutator map $ G_{s-1}\times G_0\to G_s$ induces a homomorphism $\Phi\colon G_{s-1}/G_s\times G_0/G_2\to G_s$.  Using additive notation and writing in coordinates,
$$
\text{for }1\leq i\leq r,\quad \big(\Phi(x,y)\bigr)_i=
\sum_{j=1}^q\sum_{k=1}^p\Phi_{i,j,k} x_jy_k.
$$
Since the commutator map takes $(G_{s-1}\cap \Gamma)\times \Gamma_0$ to $G_s\cap\Gamma$, 
it follows that the coefficients  $\Phi_{i,j,k}$ are integers.  

We remark that for $g\in G_{s-1}$, we have that $[g,\tau]=[g,\tau_0].[g,\gamma_0]$. The commutator map $g\mapsto[g,\tau]\colon G_{s-1}\to G_s$ induces a homomorphism $\Psi\colon G_{s-1}/G_s\to G_s$, with (using multiplicative notation) $\Psi(x)=\Phi(x,\tau)F(x)$.
In coordinates  (using additive notation), 
$$
\text{for }1\leq i\leq r,\quad\big(\Psi(x)\bigr)_i=\sum_{j=1}^q\bigr(F_{i,j}+\sum_{k=1}^p \Phi_{i,j,k}\beta_k\bigr)x_j.
$$

Let $\pi\colon G_s \mapsto G_s/(\Gamma \cap G_s)=\T^r$. We have that $f(G_{s-1})$ is the range of $\pi\circ\Psi$.  If this range is not dense in $\T^r$, then it is included in a proper subtorus, and there exist  integers $\lambda_1,\dots,\lambda_r$, not all equal to $0$, such that 
the range of $\Psi$ is included in the group $H$ defined by
$$
z\in H\text{ if and only if }\sum_{i=1}^r\lambda_iz_i\in\Z.
$$
In coordinates,
$$
\text{for every }x\in \R^q,\quad \sum_{i=1}^r\lambda_i\sum_{j=1}^q \bigr(F_{i,j}+\sum_{k=1}^p \Phi_{i,j,k}\beta_k\bigr)x_j\in\Z
$$
and thus
\begin{equation}
\label{eq:sousgroupe}
\text{for }1\leq j\leq q,\quad 
\sum_{i=1}^r\lambda_i\bigr(F_{i,j}+\sum_{k=1}^p \Phi_{i,j,k}\beta_k\bigr)=0.
\end{equation}
Since the coefficients $F_{i,j}$ are integers,
$$
\text{for }1\leq j\leq q,\quad 
\sum_{k=1}^p\bigl(
\sum_{i=1}^r\lambda_i \Phi_{i,j,k}\bigr)\beta_k\in\Z.
$$
Since the coordinates $\beta_k$ of $\beta$ are rationally independent, it follows that
\begin{equation}
\label{eq:sousgroupe2}
\text{for }1\leq j\leq q\text{ and } 1\leq k\leq p,\quad
\sum_{i=1}^r\lambda_i \Phi_{i,j,k}=0.
\end{equation}

This means that the range of $\Phi$ is included in  the proper closed subgroup $H$ of $G_s=\R^r$, and thus $[G_0,G_{s-1}]\subset H$.

Furthermore, plugging~\eqref{eq:sousgroupe2} into~\eqref{eq:sousgroupe}, we have that 
$$
\text{for }1\leq j\leq q,\quad \sum_{i=1}^r \lambda_iF_{i,j}=0. 
$$
This means that the range of $F$ is included in $H$, that is, $[\gamma_0,G_{s-1}]\subset H$.

As $G=\langle \Gamma_0, G_0\rangle$ and for every $x\in G_{s-1}$ the map $g\mapsto [g,x]$ is a group homomorphism, then 
$[G,G_{s-1}]=[G_0,G_{s-1}].[\gamma_0,G_{s-1}]$ and $[G,G_{s-1}]\subset H$, a contradiction.
\end{proof}

We use this lemma to complete the topological proof:
\begin{proof}[Proof of  Theorem~\ref{th:rec-nil}]
We proceed by induction on $s$. If $s=1$, there is nothing to prove. Henceforth we assume that $s\geq 2$ and that the statement holds for $(s-1)$-step nilsystems.  
 Let  $R$ be a set of Bohr recurrence and
let $(X=G/\Gamma,T)$ be a minimal $s$-step nilsystem that is not an $(s-1)$-step nilsystem;  we maintain the  notation used in Lemma~\ref{lem:densecommut}.

Let $X_0$ denote the connected component of $e_X$ in $X$. Then there exists $k\in\N$ such that $T^kX_0=X_0$, and the system $(X_0,T^k)$ is a minimal $s$-step nilsystem. On the other hand, the set $R_0=\{n\in\N\colon kn\in R\}$ is a set of Bohr recurrence. Substituting $X_0$ for $X$ and $R_0$ for $R$, we reduce to the case that $X$ is connected.  We can assume without loss that $G_0$ is simply connected.  

Let $U$ be a nonempty open subset of $X$; we want to show that there exists $n\in R$ such that $U\cap T^{-n}U\neq \emptyset$.  Without loss, we can assume that $U$ is the open ball $B(e_X,3\ve)$ centered at $e_X$ and of radius $3\ve$ for some $\ve>0$.

Let $\pi\colon X\to\wt X$ be the factor map defined just after~\eqref{eq:Xtilde}. Since $(\wt X,\wt T)$ is an $(s-1)$-step nilsystem, it follows from the induction hypothesis that there exist arbitrarily large $n\in R$ with $\pi\inv\bigl(B(e_X,\ve)\bigr)\cap \wt T^{-n}\pi\inv\bigl(B(e_X,\ve)\bigr)\neq \emptyset$.  It follows that for these values of $n$, there exist $x\in X$ and $v\in G_s$ with
$d_X(x,e_X)<\ve$ and $d_X(T^nx,v\cdot e_X)<\ve$. Lifting $x$ to $G$, we obtain $g\in G$ and $\gamma\in\Gamma$ with
$$
d_G(g,1_G)<\ve\text{ and }d_G(\tau^ng,v\gamma)<\ve.
$$
We claim that it  suffices to show that if $n$ is sufficiently large, there exists $h\in G_{s-1}$  and $\theta\in G_s\cap\Gamma$ such that  
\begin{equation}
\label{eq:htaun}
d_G(h,1_G)<\ve\text{ and }d_G([h\inv,\tau^n], v\inv\theta)<\ve.
\end{equation}
To see this, writing $y=h\cdot x$, we have that $y$ is the projection of $hg$ in $X$ and that
$$d_X(y,e_X)\leq d_G(hg,1_G)\leq d_G(h,1_G)+d_G(g,1_G)<2\ve.$$
Furthermore, 
\begin{multline*}
d_X(T^ny,e_X)\leq d_G(\tau^n hg, \theta\gamma)=d_G(h[h\inv,\tau^n]\tau^ng,\theta\gamma)\\
\leq \ve+d_G([h\inv,\tau^n]\tau^ng,\theta\gamma)
=\ve+  d_G(\tau^ng[h\inv,\tau^n],\theta\gamma)\\
\leq
2\ve+d_G(v\gamma [h\inv,\tau^n],\theta\gamma)
=2\ve+d_G([h\inv,\tau^n]v\gamma,\theta\gamma)\\
=2\ve+d_G([h\inv,\tau^n], v\inv\theta)<3\ve, 
\end{multline*}
where we used the right invariance of the distance $d_G$,  the fact that $[h\inv,\tau^n]\in G_s$, and that $G_s$ is included in the center of $G$. This proves the claim.

We are left with finding $h\in G_{s-1}$  and $\theta\in G_s$ satisfying~\eqref{eq:htaun}. Let $C$ be as in Lemma~\ref{lem:densecommut} applied with $\varepsilon$ and $v^{-1}$.
There exist $h'\in G_{s-1}$ and $\theta\in G_s\cap\Gamma$ such that $d_G(h',1_G)<C$ and $d_G([h',\tau],v^{-1}\theta)<\ve$. Since $G_s$ is isomorphic to $\R^r$, there exists $h\in G_s$ with $h^{-n}=h'$ and $d_G(h,1_G)\leq d_G(h',1_G)/n<C/n$. If $n\in R$ is larger that $C/\ve$, we have  that $d_G(h,1_G)<\ve$. Since $h\in G_{s-1}$, we have
$[h\inv,\tau^n]=[h\inv,\tau]^n=
[h^{-n},\tau]=[h',\tau]$ and $h$ satisfies the announced properties.
\end{proof}

In this proof, we actually showed that  for every small open subset $U\subset X$ and all sufficiently large $n$, $T^{-n}U$ almost contains a fiber of the projection $\pi\colon X\to \tilde{X}$.  This leads to a natural question: is there a way to formulate such a dilation property that can be used to prove the analog of Theorem~\ref{th:rec-nil} for more general systems? 

\section{Multiple recurrence in nilsystems}
\label{sec:mmultiple}

\subsection{\nildbohrzero sets}

For multiple recurrence in nilsystems, \nildbohrzero sets, introduced in~\cite{NilBohr}, play the role played by \bohrzero sets for recurrence in compact abelian groups; the complex exponentials are replaced by \emph{nilsequences} or by generalized polynomials.  In this section we adapt results of Huang, Song, and Ye~\cite{HSY} for our purposes.

\begin{definition}[see \cite{NilBohr}]
Let $s\geq 1$ be an integer.
The set $E \subset \N$ is a  \nildbohrzero set if there exist an $s$-step nilsystem $(X,T)$, $x_0\in X$, and an open neighborhood $U\subset  X$ of $x_0$ such that 
$$
\{n\in\N\colon T^nx_0\in U\} \subset  E.
$$
Note that in this definition we can restrict without loss to the case that $(X,T)$ is minimal.

A set $R\subset \N$ is a  \nildbohrzeroo set if it has nonempty intersection with all \nilbohrzero s sets.  
\end{definition}

\begin{theorem}[Huang, Song, and Ye~\mbox{\cite[Theorem A]{HSY}}]
\label{th:HSYA}
Let $s\in\N$. If $E\subset \N$  is a \nildbohrzero set, then there exist a minimal $s$-step nilsystem $(X,T)$ and a nonempty open set $U\subset  X$ such that 
$E\supset N^s(U)$.
\end{theorem}

We use this to show: 
\begin{corollary}
\label{cor:mmultiple-rec-nil}
Let $s\in \N$. For $R\subset \N$, the following are equivalent:
\begin{enumerate}
\item
\label{it:d-recurr-d-step}
$R$ is a set of $s$-recurrence for minimal $s$-step nilsystems;
\item
\label{it:strong-recur-d-step}
$R$ is a set of pointwise recurrence for minimal $s$-step nilsystems;
\item
$R$ is   a \nildbohrzeroo set.
\end{enumerate}
\end{corollary}
If $R$ satisfies any of these three equivalent conditions, then $R$ is actually a set of multiple pointwise  recurrence for  minimal $s$-step nilsystems.  Moreover, in this case, properties~\eqref{it:d-recurr-d-step} and~\eqref{it:strong-recur-d-step} remain valid for non-minimal $s$-step nilsystems, as the closed orbit of any point is  a minimal $s$-step nilsystem.
\begin{proof}
By Theorem~\ref{th:HSYA}, every set of $s$-recurrence for minimal $s$-step nilsystems is a \nildbohrzeroo set. 

By definition, \nildbohrzeroo sets are exactly sets of pointwise recurrence for minimal $s$-step  nilsystems.

Since every minimal $s$-step nilsystem is distal, it follows from the proof of Proposition~\ref{prop:mult-strong-distal} that a set of pointwise  recurrence for this class of systems  is also a set of multiple pointwise  recurrence for these systems and this implies~\eqref{it:d-recurr-d-step}.
 \end{proof}

We summarize what this means.  Let $s,\ell\geq 1$ be integers and let $R\subset  \N$. 
If $s\leq\ell$, the set $R$ is a set of $\ell$-recurrence for (minimal) $s$-step nilsystems if and only if it is a set of $s$-recurrence for (minimal) $s$-step nilsystems if and only if it is a \nildbohrzeroo  set.

However, we do not know what happens for $s>\ell$, other than for $\ell=1$:  if $R$ is a set of Bohr recurrence then it is a set of recurrence for all minimal nilsystems (Theorem~\ref{th:rec-nil}).  
As Bohr recurrence is equivalent to multiple Bohr recurrence, the multiple 
analog of Katznelson's question (Question~\ref{q:katznelson}) is easily seen to be false.  
However, we conjecture: 
 
\begin{conjecture}
\label{conj:multiple}
Let $s\geq 1$ and let $R$ be a set of $s$-recurrence for $s$-step nilsystems.  
Then $R$ is a set of  $s$-recurrence for all $t$-step nilsystems for any $t\geq s$.  
\end{conjecture}
For $s = 1$, this is the content of Theorem~\ref{th:rec-nil}.  For $s > 1$, 
the conjecture is supported by explicit computations in the affine case: a set of $s$-recurrence 
for affine $s$-step systems is also a set of $s$-recurrence for any $t$-step affine system with $t\geq s$, 
and this is carried out in Section~\ref{sec:affine}.  However, we do not know how to carry out these computations for a general nilsystem, but believe that some analog of the 
topological proof of Theorem~\ref{th:rec-nil} should be possible.

\subsection{Multiple recurrence and regionally proximal relations}
\label{sec:mult-RP}

Let $s\geq 1$ be an integer. The regionally proximal relation $\RPs(X,T)$ introduced in~\cite{HKM} for minimal systems $(X,T)$ generalizes the regionally proximal relation of Auslander~\cite{Aus}. In~\cite{HKM} we showed that  the relation  $\RPs(X,T)$ is the identity if and only if the system is a system of order $s$, meaning it is an inverse limit of $s$-step nilsystems; assuming in addition that the system is distal, this relation is an equivalence relation and the quotient is the maximal factor of order $s$ of $X$. The assumption of distality was removed in~\cite{SY}.

Many results of this section are implicit or explicit in the work of Huang, Shao, and Ye~\cite{HSY}. We extract,  rephrase, and adapt them here for completeness and our purposes, so as to give a framework for constructing explicit examples of  sets of recurrence.

\begin{notation}
We write $E(X,T)$ for the Ellis semigroup of  the system $(X,T)$.
\end{notation}

The next lemma appears as a comment in~\cite[page 71]{Aus}, but also can be deduced from the more general Theorem 15 in Chapter 7 in the same book.

\begin{lemma}
\label{lem:minimal-orbits}
Let $(X,T)$ be a minimal system and $(Z,S)$ be a distal system. Then each closed $(T\times S)$-orbit in $X\times Z$ is minimal.
\end{lemma}
\begin{proof}
Let $W$ denote the closed orbit of $(x_0,z_0)$ in $X\times Z$ under $T\times S$. The projection of $W$ on $Z$ is transitive  and thus is minimal by distality. Therefore, without loss we can assume that $(Z,S)$ is minimal.

Let $\sigma$ denote the transformation $p\mapsto S\circ p$ of $E(Z,S)$. Then $(E(Z,S),\sigma)$ is minimal. Let $K$ be a minimal $(T\times\sigma)$-invariant subset of $X\times E(Z,S)$.  The projection $K\to X$ is onto and there exists $p_0\in E(Z,S)$ such that $(x_0,p_0)\in K$. Then $p_0$ is a bijection of $Z$ and there exists $z_1\in Z$  such that $p_0(z_1)=z_0$. The image of $K$ under the map $(x,p)\mapsto (x_0,p(z_1))$ is a closed minimal $(T\times S)$-invariant subset of $X\times Z$ and contains $(x_0,z_0)$, and thus is equal to $W$.
\end{proof}

\begin{lemma}
\label{lem:mini-order}
Let $(X,T)$ be a minimal system and $(x_0,x_1)\in\RPs(X,T)$. Let $(Z,S)$ be a minimal system of order $s$ and $W$ be a closed $(T\times S)$-invariant subset of 
$X\times Z$. Then for $z\in Z$, we have $(x_0,z)\in W$ if and only of $(x_1,z)\in W$. 
\end{lemma}
\begin{proof}
By Lemma~\ref{lem:minimal-orbits}, without loss we can assume that $W$ is minimal.

Let $z_0\in Z$ be such that  $(x_0,z_0)\in W$.  We claim that $(x_1,z_0)\in W$. 

As in Lemma~\ref{lem:minimal-orbits}, $\sigma\colon E(Z,S)\to E(Z,S)$ denotes the map $p\mapsto S\circ p$ and $(E(Z,S),\sigma)$ is minimal.
Let $K$ be a closed and minimal $(T\times\sigma)$-invariant subset of $X\times E(Z,S)$. The first projection $\pi_1\colon K\to X$  is a factor map and thus by~\cite{SY}, the map
$\pi_1\times\pi_1$ maps $\RPs(K)$ onto $\RPs(X)$ and there exists $p_0,p_1\in E(Z,S)$ such that 
$$
(x_0,p_0)\in K,\ (x_1,p_1)\in K,\text{ and }\bigl((x_0,p_0),(x_1,p_1)\bigr)\in \RPs(K).
$$
Let $z_2\in Z$ be such that $p_0(z_2)=z_0$. The map $\pi\colon(x,p)\mapsto (x,p(z_2))$ from $K$ to $X\times Z$ satisfies $\pi\circ(T\times\sigma)=(T\times S)\circ\pi$ and thus its range  is a minimal $(T\times S)$ invariant subset of $X\times Z$.  This set  contains $(x_0,z_0)$ and thus is equal to $W$. 

Let $z_1=p_1(z_2)$. We have $(x_1,z_1)=\pi(x_1,p_1)\in W$.

On the other hand, $\pi\times\pi$ maps $\RPs(K)$ to $\RPs(W)$ and thus 
$$\bigl((x_0,z_0),(x_1,z_1)\bigr)\in\RPs(W).$$  
Since the second projection 
$(x,z)\mapsto Z$ is a factor map from $W$ to $Z$, $(z_0,z_1)\in\RPs(Z)$. Since $Z$ is a system of order $s$, $z_0=z_1$ and $(x_1,z_0)\in W$ and the claim is proven.  

Exchanging the roles of $x_0$ and $x_1$, we have the equivalence.
\end{proof}

\begin{lemma}
\label{lem:RP-nil-Synd}
Let $(X,T)$ be a minimal system, $(x_0,x_1)\in\RPs(X,T)$, and $U$ be an open neighborhood of $x_1\in X$. Then for every \nildbohrzero set $E$, $N(x_0,U)\cap E$ is syndetic.
\end{lemma}
\begin{proof}
Let $(Z,S)$ be a minimal system of order $s$, $z_0\in Z$, and $V$ be an open neighborhood of $z_0\in Z$. We have to show that $N(x_0,U)\cap N(z_0,V)$ is syndetic.

Let $W$ be the closed $(T\times S)$-orbit of $(x_0,z_0)$  in $X\times Z$. By Lemma~\ref{lem:minimal-orbits}, $(W,T\times S)$ is minimal and, by Lemma~\ref{lem:mini-order}, $(x_1,z_0)\in W$. We have that 
$(U\times V)\cap W$ is an open neighborhood of $(x_1,z_0)$ in $W$ and thus 
$N_{T\times S}\bigl((x_0,z_0),U\times V\bigr)=N_T(x_0,U)\cap N_S(z_0,V)$ is syndetic.
\end{proof}

\begin{theorem}[Huang, Song, and Ye~\mbox{\cite[Theorem E]{HSY}}]\label{theorem:Nxu-rec}
Let $(X,T)$ be a minimal system, $(x_0,x_1)\in\RPs(X,T)$, and $U$ be an open neighborhood of $x_1\in X$. Then $N(x_0,U)$ is a set of $s$-recurrence.
\end{theorem} 
\begin{proof}
Let $(X,T)$ be a minimal system, $V$ be a nonempty open subset of $X$, and 
let $\mu$ be an invariant ergodic measure on $X$. By~\cite[Theorem A(2)]{HSY}, 
there exist a \nildbohrzero set  $E$ and a set of uniform upper density zero $F$ such that $N^s(V)\supset E\setminus F$.
By Lemma~\ref{lem:RP-nil-Synd}, $E\cap N(x_0,U)$ is syndetic and thus is not included in $F$; it follows that $N^s(V)\cap N(x_0,U)\neq\emptyset$.
\end{proof}

We use this to construct explicit examples of various sets of recurrence:
\begin{example} 
If $(X,T)$ is a minimal system and $x_0, x_1\in X$ are 
proximal, then $(x_0, x_1)\in\RPs(X,T)$ for every $s\in\N$ (see ~\cite{HKM}).  
Thus if $U$ is an open neighborhood of $x_1$, then $N(x_0,U)$ is a set of multiple recurrence. 
If $x_0\notin\overline{U}$, then this set is not a set of 
pointwise recurrence.  
\end{example}

In Frantzikinakis, Lesigne, and Weirdl~\cite{FLW}, the authors build examples of sets of  $s$-recurrence that are not sets of $(s+1)$-recurrence; the framework is measurable dynamics but the same constructions  also work in the topological setting. We give a more general framework that gives further insight into the behavior of these examples using Theorem \ref{theorem:Nxu-rec}. 

\begin{corollary}
\label{cor:Nxrecir}
Let $(X,T)$ be  a minimal $s$-step nilsystem and let $(x_0,x_1)\in \RP^{[s-1]}(X,T)$. Let $U$ be an open neighborhood of $x_1$ in $X$ with $x_0\notin \overline U$. 
Then $N(x_0,U)$ is a set  of $(s-1)$-recurrence, is not a set of $s$-recurrence (even for $s$-step nilsystems), and is 
not a set of pointwise  recurrence.
\end{corollary}

\begin{proof}
Since $(x_0,x_1)\in\RP^{[s-1]}(X,T)$, by Theorem~\ref{theorem:Nxu-rec}, $N(x_0,U)$ is a set of $(s-1)$-recurrence.  On the other hand, $N(x_0,X\setminus\overline U)$ is a \nilbohrzero s set that does not intersect $N(x_0,U)$, and thus this last set is not a \nilbohrzeroo s set.  By Corollary~\ref{cor:mmultiple-rec-nil}, it is not a set of $s$-recurrence and is not a set of pointwise  recurrence.
\end{proof}

This leads to various examples of sets of recurrence and non-recurrence.  
We begin with a simple observation.  If $(X,T)$ is a minimal $2$-step nilsystem and 
$Y$ is its maximal equicontinuous factor, then $X$ is an isometric extension of $Y$. If $x_0, x_1\in X$ are distinct points with the same projection in $Y$, then $(x_0,x_1)\in \RP^{[1]}(X,T)$. Thus if $U$ is an open neighborhood of $x_1$ and 
$x_0\notin\overline{U}$, then $N(x_0, U)$ is a set of Bohr recurrence and thus of 
Bohr multiple recurrence. However, it is not a set of double recurrence for $(X,T)$.  

More generally, we have the examples from~\cite{FLW}:
\begin{example}
\label{example:FLW}
The set $S=\{n\in\N\colon\norm{n\beta}>\ve\}$ is not a set of recurrence for any $\beta\in\T$ 
and $0< \ve <  1/2$.  More generally, it was shown in~\cite{FLW} that for any $s\geq 1$, $\ve > 0$, and any $\beta\in\T$,  the set $S=\{n\in\N\colon \norm{n^s\beta}> \ve\}$ is a set of $(s-1)$-recurrence and is not a set of $s$-recurrence.  

We explain how to prove this result using the current machinery.  
For rational $\beta$ the result is obvious and so we assume that $\beta$ is irrational.  Let $(\T^s, T)$ be the $s$-step affine nilsystem defined in Section~\ref{sec:affine}, where $Tx = Mx+\alpha$ and $\alpha$ is to yet be determined. 
Set $a = (0,0, \ldots, 0) \in\T^s$ and $b = (1/2, 0, \ldots, 0)\in\T^s$.  
Then for every $n\in\N$, we have that 
$T^na = (\id + M +\ldots + M^{n-1})\alpha$. By formula~\eqref{eq:M} giving the entries of $M^n$, we can choose $\alpha$ with $\alpha_s$ irrational such that $(T^na)_1 = n^s \beta$.  
As in the proof of Theorem~\ref{th:lift-affine},
the maximal $(s-1)$-step factor of $(\T^s, T)$
 is $(\wt{X}, \wt{T})$, where $\wt X$ is the quotient of $\T^s$ under the subgroup $\{(t, 0, \ldots, 0)\colon t\in\T\}$. Thus $a$ and $b$ have the same projection on $\wt X$ and so $(a, b)\in\RP^{[s-1]}(\T^s,T)$.  
Setting $U = \{x\in\T^s\colon \norm{x_1}< \ve\}$, 
we have that $U$ is an open set containing $a$ and $b\notin\overline{U}$. 
The statement now follows from Corollary~\ref{cor:Nxrecir}.

On the other hand, the set $\{n \in \N \colon \norm{n^s\beta}< \ve\}$ is exactly 
$N(a, U)$, and so as already remarked, it is a set of multiple recurrence. 
\end{example}

We remark that all of these examples are large sets, in the sense that they have positive density. However, there are many examples of sets of multiple recurrence of density zero, such as any $\ip$-set~\cite{furstenberg} the set of values of a polynomial~\cite{BL},  the set of shifted primes~\cite{FHK}, or a set containing arbitrarily long arithmetic progressions and 
such that any integer occurs as a common difference~\cite{FWierdl}.   Adapting ideas of~\cite{FWierdl}, one can construct zero density sets of $(s-1)$-recurrence that are not sets of $s$-recurrence.

\subsection{Lifting multiple recurrence in affine systems}

\label{sec:affine}
\begin{definition}
For $s>1$, let 
$M$ be an $s \times s$ matrix with integer entries.  Assume that $M$ is {\em unipotent}, meaning that $(M-\id)^s=0$, and let $\alpha\in\T^s$. Define $T\colon  \T^{s} \to \T^{s}$ by $T(x)=Mx+\alpha$ (operations are always $\mod 1$).  The system $(\T^s, T)$ is called an 
{\em affine system on $\T^s$}.
\end{definition}

The system $(\T^s, T)$ is minimal
if the projection of $\alpha$ on $\T^s/\ker(M-\id)$ generates a minimal rotation on this torus~\cite{parry}.  

The system $(\T^s, T)$ can be represented as a nilsystem.  Namely, let $G$ denote the group of transformations of $\T^s$ spanned by $M$ and the translations $S_\beta\colon x\mapsto x+\beta$ for $\beta\in\T^s$ and let $\T^s$ be identified with the subgroup 
$\{S_\beta\colon\beta\in\T^s\}$ of $G$. For $j\geq 2$, $G_j\subset\T^s$ and more precisely
 $$
G_j =\range(M-\id)^{j-1}.
$$
Therefore,  $G$ is an $s$-step nilpotent Lie group, and the stabilizer of $0$ is $\Gamma=\{M^n\colon n\in\Z\}$. Then $\T^s$ is identified in the natural way with $G/\Gamma$.

We prove Conjecture~\ref{conj:multiple} for affine systems: 
\begin{theorem}
\label{th:lift-affine}
Let $1 \leq r \leq s-1$ be an integer. If $R \subset  \N$ is a set of $r$-recurrence for all $(s-1)$-step minimal affine systems, then it is a set of $r$-recurrence for $s$-step minimal affine systems.
\end{theorem}

Before proving the theorem, we start with some preliminary simplifications.  There is a change of basis such that $M = PM'P^{-1}$, where $P$ has integer 
entries and non-zero determinant, and such that $M'$ 
is in Jordan canonical form. Let $\alpha'$ be such that $P\alpha' = \alpha$ and define $T'\colon \T^s\to\T^s$ to be $T'x = M'x+\alpha'$.  Then the system $(\T^s, T')$ is a minimal affine nilsystem, and the map $P\colon \T^s\to\T^s$ is a finite to one factor map from this system to the system $(\T^s, T)$.  Thus it suffices to prove the theorem for a system whose matrix $M$ is in Jordan canonical form.  

Furthermore, for notational convenience, we restrict ourselves to the case that there is a single block in the Jordan form, and we note at the end of the proof how to generalize this for multiple blocks. Thus, henceforth we assume that $M$ is an integer matrix with 
$$
M_{i,j} = 
\begin{cases}
1 & \text{ if } j=i \text{ or } j=i+1 \\
0 & \text{ otherwise. }
\end{cases}
$$
For $2\leq j\leq s$, we have
$$
G_j=\range(M-\id)^{j-1}=\bigl\{ x=(x_1,\dots,x_s)\in\T^s\colon x_i=0\text{ for }i\geq s-j+2\bigr\}.
$$
Minimality of the system $(\T^s, T)$ is equivalent to the last coordinate $\alpha_s$ of $\alpha$ being irrational.  

The matrix $M$ is exactly $s$-unipotent, meaning that 
$$
(M-\id)^s=0\text{ and }(M-\id)^{s-1}\neq 0, 
$$
and 
the system $(\T^s, T)$ is exactly an $s$-step minimal 
nilsystem, meaning that it is not an $(s-1)$-step nilsystem. 

From the form of the matrix $M$, we deduce that
for any $n\in\N$, the entries of $M^n$ satisfy
$M^n_{i,i}=1$ for $1\leq i\leq s$ and 
\begin{equation}
\label{eq:M}
\text{for }1\leq i< j\leq s,\  M^n_{i,j}=p_{j-i}(n), 
\end{equation}
where $p_1(n)=n$ and, for $1\leq k<s$, $p_{k}$ is a polynomial with integer coefficients whose degree is exactly $k$ such that $p_{k}(0)=0$. 

\begin{notation} 
Throughout this proof, $C$ denotes some constant, possibly taking on different values, 
where the only dependence is on $s,M$ and $\alpha$, but not on $n$.
For other  objects, the dependence on $n$ is often left implicit.
\end{notation}

\begin{lemma}
\label{lem:A}
Let $n \in \N$. For every $y\in\R^s$ with $y_s=0$, there exists a unique  $x\in \R^s$ such that $x_1=0$ and $(M^n-\id)x=y$ and we write $x=Ay$. Furthermore, if for some $k\in\{3,\dots,s\}$ we have $y_i=0$ for $i\geq s-k+2$, then $x_j=0$ for $j\geq s-k+3$.

Finally, there exists a constant $C>0$ such that if for some constant $\kappa>0$ we have $|y_i|\leq \kappa/n^{i-1}$ for $1\leq i\leq s-1$, then $|x_j|\leq C\kappa/n^{j-1}$ for $2\leq j\leq s$.
\end{lemma}

It is immediate that the map $A$ is linear, given by a matrix with integer entries, and we can view it also as a homomophism from $
G_2$ to $\T^s=G$, mapping $G_k$ to $G_{k-1}$ for $3 \leq k\leq s$.

\begin{proof}
A real vector $x=(0,x_2,\dots,x_s)$ satisfies $(M^n-\id)x=y$ if and only if $x_2,\dots,x_s$ satisfy the linear system:
\begin{align*}
y_1 &= p_{1}(n)x_2+\dots+p_{d-1}(n)x_s;\\
\vdots & = \vdots\\
y_i &=p_{1}(n)x_{i+1}+\dots+p_{d-i}(n)x_s;\\
\vdots &=\vdots\\
y_{s-2} &= p_{1}(n)x_{s-1}+p_{2}(n)x_s;\\
y_{s-1} &= p_{1}(n)x_s.
\end{align*}
This triangular system has a unique solution since the coefficients $p_{1}(n)$ are non-zero and the first statement follows.  The second statement is obvious.
 
Since for $1\leq k \leq s$ the polynomial $p_{k}$ is exactly of degree $k$ and satisfies $p_k(0)=0$, there exist constants $C_1,C_2>0$ such that 
$$
\frac {C_1}{n^{k}}\leq p_{k}(n)\leq \frac {C_2}{n^{k}}\text{ for all }n\in\N
$$
and the last statement follows.
\end{proof}

We use this to complete the proof of Theorem~\ref{th:lift-affine}: 
\begin{proof}[Proof of Theorem~\ref{th:lift-affine}]
We proceed by induction on $s$. 
Assume that $r < s$ and that $R$ is  a set of $r$-recurrence for all affine $(s-1)$-step nilsystems. 
We show that $R$ is a set of $r$-recurrence for any affine $s$-step nilsystem.

Let $(\wt X,\wt T)$ be defined as in~\eqref{eq:Xtilde}. Recall that $\wt X$ is the quotient of $X$ under the action of $G_{s}$, and so can be identified with $G/G_s=\T^{s-1}$.  Then $(\wt X,\wt T)$ is an $(s-1)$-step affine nilsystem, given by the matrix $\wt M$ induced by $M$, and the translation by $\wt \alpha$, the image of $\alpha$ in $\T^{s-1}$.

Since $1\leq r\leq s-1$, by the induction hypothesis,   there exist arbitrarily large  $n\in R$ and $\wt x\in\T^{s-1}$ with $\norm{\wt x}\leq\ve$ and $\norm{\wt T^{kn}\wt x}\leq\ve$ for $1\leq k\leq r$. Lifting to $X$, 
there exist 
$x\in\T^{s}$ and $w_1,\dots,w_r\in G_s$ with $w_k = (w_{k,1}, \ldots, w_{k,s})$ for $j=1, \ldots r$ such that
$$
\norm x\leq \ve\text{ and }\norm{T^{kn}x- w_k}\leq \ve\text{ for }1\leq k\leq r.
$$

We need to show that if $n$ is sufficiently large, there exists $y\in G_{s-r}$ such that 
$$
\norm y\leq C\ve \text{ and }\norm{T^{k n}(x+y)}\leq C\ve \text{ for }1\leq k\leq r.
$$

For any $k\in\N$, we have that $T^{k n}(x+y)=T^{k n}x+M^{k n}y$ and so the system of approximate equations to be solved is 
\begin{gather*}
 y\in G_{s-r},\ 
\norm y\leq C\ve;\\
\norm{M^{k n}y+w_k}\leq C\ve\text{ for }1\leq k\leq r.
\end{gather*}
Set $v_k = -\sum_{j=1}^k \binom{k}{j} (-1)^{k-j} w_j\in G_s$ for $k=1, \ldots, r$. 
Then the system we need to solve becomes
\begin{gather}
\label{eq:sys0} y\in G_{s-r};\\
\label{eq:sys1} \norm y\leq C\ve;\\
\label{eq:sys2}
\norm{(M^{n}-\id)^k y-v_k}\leq C\ve \text{ for }1\leq k\leq r.
\end{gather}

By Lemma~\ref{lem:A} and by induction on $\ell$, for $1\leq k\leq r$ and  $1\leq \ell\leq s-1$, the   elements $A^\ell v_k$ satisfy
\begin{gather}
\notag
(M^k-\id)^\ell A^\ell v_k=v_k;\\
\notag
A^\ell v_k\in G_{s-\ell};\\
\label{eq:lemA}
(A^\ell v_k)_1=0\text{ and }|(A^\ell v_k)_i|\leq C/n^{i-1}\text{ for }2\leq i\leq s.
\end{gather}
Define $y_k\in G$ by 
$$
y_k=A^k v_k\text{ for }1\leq k \leq r\text{ and }y=y_1+\dots+y_r.
$$
Then we claim that for $n$ sufficiently large, $y$ satisfies conditions~\eqref{eq:sys0}, \eqref{eq:sys1} and~\eqref{eq:sys2}.

To see this, by construction, for $1\leq k \leq r$, $y_k\in G_{s-k}\subset G_{s-r}$ and~\eqref{eq:sys0} is satisfied. 

By~\eqref{eq:lemA}, all coordinates of $y_k$ are bounded in absolute value by $C/n$, and thus $\norm{y_k}\leq C/n$. It follows that $\norm y\leq C/n$ and that~ \eqref{eq:sys1} is satisfied when $n$ is sufficiently large.

Furthermore, for $1\leq k\leq r$,
\begin{multline*}
(M^n-\id)^k y \\
\begin{aligned}
=&\sum_{\ell=1}^{k-1} (M^n-\id)^k y_\ell &+& (M^n-\id)^k y_k&+&
\sum_{\ell=k+1}^r(M^n-\id)^k y_\ell.\\
=&S_1&+&S_2&+&S_3.
\end{aligned}
\end{multline*}
We analyze these three terms.  
For $1\leq \ell<k$, since $y_\ell\in G_{s-\ell}$, we have that $(M^n-\id)^k y_\ell=0$ and thus $S_1=0$.
By construction, $S_2=(M^n-\id)^kA^kv_k=v_k$.
 
For $k<\ell \leq r$, $(M^n-\id)^ky_\ell=(M^n-\id)^kA^\ell v_\ell=A^{\ell-k}v_\ell$ and by~\eqref{eq:lemA} all coordinates of this element are bounded by $C/n$ and thus $\norm{(M^n-\id)^ky_\ell}\leq C/n$. It follows that
$\norm{S_3}\leq C/n$, and~\eqref{eq:sys2} holds when $n$ is sufficiently large.

For generalizing to the case where there may be more blocks in the Jordan matrix $M$, we note that the proof applies for all sufficiently large $n\in R$.  Thus taking $n$ to be the maximum of these iterates, we deduce the 
general case. 
 \end{proof}
 
\begin{figure}[h]
\label{fig:main} 
\begin{displaymath}
\xymatrix{ \ar^>{{\bf \ell-{\rm recurrence}}}@<-1ex>[6,0];[]& \vdots & \vdots & \vdots & & \vdots &  \\
s-& \bullet \ar@2{->}[u]\ar@{->}[r]|{\SelectTips{cm}{}\object@{/}}|{} & \bullet \ar@2{->}[u]\ar@{->}[r] |{\SelectTips{cm}{}\object@{/}}|{}& \bullet \ar@2{->}[u] \ar@{->}[r]|{\SelectTips{cm}{}\object@{/}}|{} & \cdots  \ar@{->}[r]|{\SelectTips{cm}{}\object@{/}}|{} &  \bullet \ar@2{->}[u]\ar@2{~>}[r] & \cdots  \\
& \vdots \ar@2{->}[u]\ar@{->}[r] |{\SelectTips{cm}{}\object@{/}}|{}& \vdots \ar@2{->}[u]\ar@{->}[r] |{\SelectTips{cm}{}\object@{/}}|{}& \vdots \ar@2{->}[u]\ar@{->}[r]|{\SelectTips{cm}{}\object@{/}}|{} & \cdots  \ar@2{~>}[r] &  \iddots \ar@{->}[u]|{\SelectTips{cm}{}\object@{/}}|{}\ar@2{~>}[r] & \cdots  \\
3- & \bullet \ar@2{->}[u]\ar@{->}[r] |{\SelectTips{cm}{}\object@{/}}|{}& \bullet \ar@2{->}[u]\ar@{->}[r] |{\SelectTips{cm}{}\object@{/}}|{}& \bullet \ar@2{->}[u]\ar@2{~>}[r] & \cdots  \ar@2{~>}[r] &  \bullet \ar@{->}[u]|{\SelectTips{cm}{}\object@{/}}|{}\ar@2{~>}[r] & \cdots  \\
2-& \bullet \ar@2{->}[u]\ar@{->}[r] |{\SelectTips{cm}{}\object@{/}}|{}& \bullet \ar@2{->}[u]\ar@2{~>}[r] & \bullet \ar@{->}[u]|{\SelectTips{cm}{}\object@{/}}|{}\ar@2{~>}[r] & \cdots  \ar@2{~>}[r] &  \bullet \ar@{->}[u]|{\SelectTips{cm}{}\object@{/}}|{}\ar@2{~>}[r] & \cdots  \\
1-& \bullet \ar@2{->}[u]\ar@2{->}[r] & \bullet \ar@{->}[u]|{\SelectTips{cm}{}\object@{/}}|{} \ar@2{->}[r] & \bullet \ar@{->}[u]|{\SelectTips{cm}{}\object@{/}}|{}\ar@2{->}[r] & \cdots  \ar@2{->}[r] &  \bullet \ar@{->}[u]|{\SelectTips{cm}{}\object@{/}}|{}\ar@2{->}[r] & \cdots  \\
&\stackrel{\mid}{1} & \stackrel{\mid}{2} & \stackrel{\mid}{3} &  & \stackrel{\mid}{s}  &   \ar_>{{\bf s-{\rm step}}}@<1ex>[0,-6];[] 
}
\end{displaymath}

\caption{The horizontal axis represents the step of the nilsystem and the vertical 
axis represents the level of recurrence.  The vertical implications are proven in~Corollary \ref{cor:mmultiple-rec-nil} 
and counterexamples for vertical  implications (with step greater than recurrence) are given by Example~\ref{example:FLW}.  Horizontal 
squiggly implications are proven only for affine nilsystems 
in Theorem~\ref{th:lift-affine}, while the full horizontal implications are proven in Theorem~\ref{th:rec-nil}.  
Counterexamples for the horizontal implications (with recurrence greater than step) follow from Corollary~\ref{cor:mmultiple-rec-nil} and Corollary~\ref{cor:Nxrecir}.  
}
\end{figure}
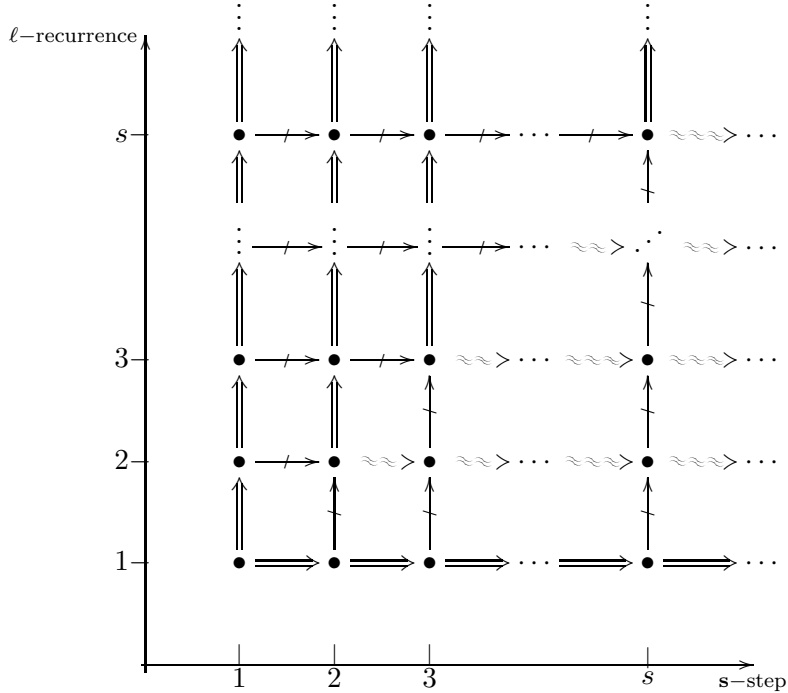

\section{The Ramsey property}
\label{sec:Ramsey}

\begin{definition}
A property is {\em Ramsey} if for any set $R\subset\N$ having this property and any partition $R = A\cup B$, 
at least one of $A$ or $B$ has this property.  
\end{definition}
The Ramsey property is also sometimes referred to as  {\em divisible}; an equivalent characterization 
is that its dual is a filter (see~\cite{furstenberg, glasner, BG}). 

The following proposition appears in several places in the literature (see for example~\cite[Proposition 7.2.4]{HSY}), but for completeness we give a proof: 
\begin{proposition}
\label{prop:ramsey}
The family of sets of $\ell$-recurrence has the Ramsey property.  
 \end{proposition}
\begin{proof}
We proceed by contradiction. Assume that $R$ is a set of $\ell$-recurrence and that $R =A\cup B$ 
is a partition such that neither $A$ nor $B$ is a set of $\ell$-recurrence.  
Thus there exist two minimal systems $(X,T)$ and $(Y,S)$ and open sets $U\subset  X$ and $V\subset  Y$ 
such that $N_T^\ell(U)\cap A=\emptyset$ and $N_S^\ell(V)\cap B=\emptyset$.  
Let $Z$ be a minimal subset of the product $X\times Y$.  By minimality of the $\Z^2$-action of $\{T^n\times S^m\colon n,m\in\Z\}$ on $X\times Y$, we can choose $n,m\in \Z$  such that 
$$Z' = (T^n\times S^m)Z\cap(U\times V)\neq\emptyset.
$$
Then $Z'$ is a nonempty open set and so by assumption, 
$N^\ell_{T\times S}(Z')\cap R\neq\emptyset$.  
But $N^\ell_{T\times S}(Z')\subset  N_S^\ell(V)\cap N_T^\ell(U)$, a contradiction that 
$R=A\cup B$. 
\end{proof}

\begin{corollary}
Let $R$ be a set of $\ell$-recurrence, $(X,T)$ a minimal system, and $U$ be a nonempty open subset of $X$. Then $R\cap N^\ell(U)$ is a set of 
$\ell$-recurrence.
\end{corollary}

\begin{proof}
By definition $\N\setminus N^\ell(U)$ is not a set of $\ell$-recurrence. Thus $R\setminus N^\ell(U)$ is not a set of $\ell$-recurrence. The result follows from Proposition~\ref{prop:ramsey}.
\end{proof}

In particular, it follows from this corollary and Theorem~\ref{th:HSYA} that if $R$ is a set of $\ell$-recurrence and $E$ is a \nildbohrzero set, then $E\cap R$ is a set of $\ell$-recurrence.

Similarly, one can easily check that 
by passing to the product, a set of pointwise recurrence for minimal, distal systems is Ramsey: 
\begin{proposition}
The family of sets of pointwise  recurrence for minimal distal systems has the Ramsey property: if $A$ and $B$ are subsets of $\N$ such that $A\cup B$ is a set of  pointwise  recurrence for minimal distal systems, then at least one of the sets $A$ or $B$
is a set pointwise  recurrence for minimal distal systems
\end{proposition}

\begin{proof}
We  assume  by contradiction that there exist two distal minimal systems $(X,T)$ and $(Y,S)$, $x\in X$, $y\in Y$ and $\ve >0$ such that
$$
\text{for every }n\in A,\ d_X(T^nx,x)\geq\ve\ ;\
\text{for every }n\in B,\ d_Y(T^ny,y)\geq\ve.
$$
Let $X\times Y$ be endowed with the sum distance $d_{X\times Y}((x,y),(x',y'))=d_X(x,x')+d_Y(y,y')$.
Since $X$ and $Y$ are distal, the closed $(T\times S)$-orbit $W$ of $(x,y)$ in $X\times Y$ is minimal. Since $A\cup B$ is a set of  pointwise recurrence for minimal distal systems, there exists $n\in A\cup B$ such that
$$
\ve> d_{X\times Y}((T\times S)^n(x,y),(x,y))=
d_X(T^nx,x)+d_Y(S^ny,y),
$$
a contradiction.
\end{proof}
\begin{question} Does the family of sets of pointwise (or multiple or simultaneous) topological recurrence have the Ramsey property?
\end{question}

\section{Large sets and syndetic large sets}
\label{sec:large}

\subsection{Fixing the number of colors}
In the definition of a set of recurrence, we 
consider an arbitary, finite partition of the integers 
and arithmetic progressions of arbitary length.  
Restricting the length of the progression leads to the 
definition of $\ell$-recurrence.  Instead, we can 
restrict the number of cells in the partition and 
this is the point of 
view taken in Brown, Graham, and Landman~\cite{BGL}, where this is studied from a purely combinatorial point of view.  
They define: 
\begin{definition}
If $r\geq 2$ is an integer, a set $R\subset \N$ is {\em  $r$-large} if every coloring of the integers with $r$ colors 
contains arbitrarily long monochromatic progressions with step in $R$.  The set $R\subset \N$ is {\em large} if it is $r$-large for every $r\geq 2.$
\end{definition}

Analogous to Theorem~\ref{th:rec}, this property can be described dynamically: 
\begin{proposition}
\label{prop:large2}
Let $r\geq 2$. The set $R\subset  \N$ is $r$-large if and only if 
for every system $(X,T)$, every open cover ${\mathcal U}=(U_1,\ldots, U_r)$ of $X$ by $r$ open sets, and every $\ell\geq 2$, there exist $j\in\{1,\dots,r\}$ and $n\in R$ such that $n\in N^\ell(U_j)$.
\end{proposition}

In particular, a set of integers is a set of multiple recurrence if and only if it is $r$-large for every $r$, meaning it is large.

Thus,  a question asked in~\cite{BGL} becomes: 
\begin{question}[Brown, Graham, and Landman~\cite{BGL}]
Are all $2$-large sets sets of  multiple recurrence?
\end{question}

We rephrase some of the other results from~\cite{BGL}, with some minor modifications, putting them into dynamical language.  Their example~\ref{example:FLW} becomes: 
\begin{lemma}
\label{lem:not2large}
Let $\alpha\in\T$ and $\ve>0$.
The set $S=\{n\in\N\colon\norm{n\alpha}>\ve\}$ is not $2$-large.
\end{lemma}
\begin{proof}
Let  $\alpha\in\T$,  $J_1=[0,1/2)$ and $J_2=[1/2,1)$. We define a $2$-coloring of $\N$ by $C_j=\{n\in\N\colon n\alpha\in J_j\}$ for  $j=1,2$. Let 
$\ell=1+\lceil 1/2\ve\rceil$. 
We show that there is no monochromatic progression of length $\ell$  and with common difference $n\in R$. Assume, by contradiction, that such a progression $P=\{a+in\colon 0\leq i\leq\ell-1\}$ exists.

Choose $\beta\in(-1/2,1/2]$ such that $\beta=n\alpha\bmod 1$. Without loss of generality, we can assume that $0\leq\beta\leq 1/2$ and that $P\subset  C_1$.
For $0\leq k<\ell$, let $a_k=a\alpha+k\beta\bmod 1$. Then the set $X=\{a_k\colon 0\leq k<\ell\}$ is contained in $J_1$. 
For $0\leq k<\ell-1$, we have that $a_{k+1}=a_k+\beta\bmod 1$.  On the other hand,  $0\leq a_{k}<1/2$ and $0\leq\beta\leq 1/2$ and thus $0\leq a_k+\beta<1$. We deduce that
$$
a_{k+1}=a_k+\beta\text{ for }0\leq k<\ell-1.
$$
Therefore $a_{\ell-1}=a_0+(\ell-1)\beta$, and thus  $\beta=(a_{\ell-1}-a_0)/(\ell-1)\leq 1/2(\ell-1)<\ve$, a contradiction.
\end{proof}
 
In analogy with  Proposition~\ref{prop:ramsey}, the characterization in Proposition~\ref{prop:large2}  of large sets 
leads to a dynamical proof for the following: 
\begin{proposition}[Brown, Graham, and Landman~\cite{BGL}]
\label{prop:ramsey-large}
If $r_1,r_2\geq 2$ and $S_1\cup S_2$ is $r_1r_2$-large, then some $S_i$ is $r_i$-large for $i=1,2$.
\end{proposition}
\begin{proof}
Assume not.  Instead, assume that for $i=1,2$, the set $S_i$ is not $r_i$-large and 
$\CC_i=\{C_{i,1},\ldots, C_{i,r_i}\}$ is an $r_i$-coloring of $\N$ such that there is no
progression of length $\ell_i$, with step in $S_i$ contained in atoms of $\CC_i$.
Let  $\CC_1\vee\CC_2$ be the partition $\{C_{1,j}\cap C_{2,k}\colon 1\leq j\leq r_1,\ 1\leq k\leq r_2\}$ and  set $\ell=\max(\ell_1,\ell_2)$. Then there exists a progression of length $\ell$, with step $d\in S_1\cup S_2$, that is 
monochromatic under the partition $\CC_1\vee\CC_2$. 

But if $d\in S_i$, we have a contradiction of the fact that the progression is monochromatic for $\CC_i$.
\end{proof}

However, we are unable to answer the following:
\begin{question}[Brown, Graham, and Landman~\cite{BGL}]
Does the family of $2$-large sets have the Ramsey property? 
\end{question}

\begin{proposition}
Let  $S\subset \N$ and $r\geq 2$. 
If $d\geq 1$, $E$ is a \bohrzero set of dimension $d$ and $S$ is $2^dr$-large for some $r\in\N$, then $S\cap E$ is $r$-large.
\end{proposition}
\begin{proof}
We proceed by induction on the dimension $d$ of the \bohrzero set $E$.
Assume that $E$ is a \bohrzero-set of dimension $1$. Then $E\supset E':=\{n\colon \norm{n\alpha}<\ve\}$ for some $\alpha\in\T$ and some $\ve>0$. Let $S$ be $2r$-large for some $r\in\N$. Write $S=(S\cap E)\cup(S\setminus E)$. By Lemma~\ref{lem:not2large}, the second of these sets is not $2$-large, and by Proposition~\ref{prop:ramsey-large}, $S\cap E$ is $r$-large.

Assume that $d\geq 1$, that the result holds for \bohrzero-sets of dimension $d$, and let $E$ be a \bohrzero-set of dimension $d+1$ and $S$ be a $2^{d+1}r$-large set. Then $E\supset F\cap E'$,  where 
$F$ is a \bohrzero set of dimension $d$ and 
$E'=\{n\colon \norm{n\alpha}<\ve\}$ for some $\alpha\in\T$ and some $\ve>0$. 
As above, we write $S=(S\cap E')\cup(S\setminus E')$.  Again, by Lemma~\ref{lem:not2large} the second of these sets is not $2$-large and by Proposition~\ref{prop:ramsey-large}, $S\cap E'$ is $2^dr$-large. By the induction hypothesis, 
$S\cap E\supset(S\cap E')\cap F$ is $r$-large.
\end{proof}

\subsection{$r$-large sets and nilsystems}
We are interested if the results  of Section~\ref{sec:mult-RP} have counterparts for $r$-large sets. 
For example, consider the analog of Corollary~\ref{cor:Nxrecir}:

\begin{question}
Let $(X,T)$ be a minimal $d$-step nilsystem, $x_0,x_1\in X$, and $U$ be an open neighborhood of $x_1$ with $x_0\notin\overline{U}$. Then $N(x_0,U)$ is not a set of multiple recurrence.
Does there exist some $r\geq 2$ such that  $N(x_0,U)$ is not $r$-large? 
\end{question}

This can be answered in the particular case of affine nilsystems, as the affine nilsystems  give rise to  polynomials (see Section~\ref{sec:affine}): 

\begin{proposition}
\label{prop:affine1}
Let $\ell\geq 1$, $0< \delta\leq 1/2$, $\alpha\in\T$,
\begin{equation}
\label{eq:defOmega}
R=\{n\in\N\colon \norm{ n^\ell\alpha}>\delta\}
\end{equation}
and $m=\lceil 2^{\ell-1}\delta^{-1}\rceil$. Then $R$ is not $m$-large.
\end{proposition}
\begin{proof} 
We proceed by contradiction and assume that $R$ is $m$-large. 

To avoid ambiguity, we stress that we consider here $\alpha$ as an element of $\T=\R/\Z$. We write $\underline\alpha$ for the real in $(-1/2,1/2]$ such that $\alpha=\underline\alpha\bmod 1$, $\underline\beta=\underline\alpha/\ell!\in\R$ and $\beta=\underline\beta\bmod 1\in\T$.

Let $\T=I_1\cup\ldots\cup I_m$ be a partition of $\T$ in (half open) intervals of length  $1/m$.  For $1\leq j\leq m$, 
let $C_j=\{p\in\N\colon p^\ell\beta\in I_j\}$.  By hypothesis, there exists an arithmetic progression $P=\{ a+pn\colon 0\leq p\leq\ell\}$  of length $\ell+1$, with step $n\in R$, and it is monochromatic under this coloring, meaning that there exists $j$, $1\leq j\leq m$, such that   $(a+pn)^\ell\beta\in I_j$ for $0\leq p\leq\ell$.

If $(u(p))$ is a  sequence of reals, write $(\Delta u)(p)=u(p+1)-u(p)$. Iterating this definition, we have that
$$
(\Delta^\ell u)(p)=\sum_{k=0}^\ell\binom \ell k(-1)^ku_{p+k}.
$$
Using this with $u_p=(a+pn)^\ell\underline \beta$, 
$$
n^\ell\underline\alpha=\ell!n^\ell\underline\beta=(\Delta^\ell u)(0)=\sum_{k=0}^\ell\binom \ell k(-1)^ku_{k}.
$$
For every real number $x$, we write $\{x\}$ for the difference between $x$ and the nearest integer. For  $0\leq p\leq \ell$, the points $u_p\bmod 1$ belong to the same half open interval $I_j$ of length $1/m$, and thus for $0\leq p\leq\ell-1$ we have 
$\{(\Delta u)(p)\}=\{ u_{p+1}-u_p\}\in(-1/m,1/m)$. By the same argument and using induction, $\{(\Delta^\ell u)(0)\} \in (-2^{\ell-1}/m,2^{\ell-1}/m)$, meaning that $\{n^{\ell}\underline\alpha\}\in (-2^{\ell-1}/m,2^{\ell-1}/m)$.  Thus $\norm{n^\ell\alpha}< 2^{\ell-1}/m<\delta$, a contradiction.
\end{proof}

\subsection{Syndetic large sets}

\begin{definition}
Recall that a set $E\subset \N$ is syndetic if there exists $r\geq 1$ such that every interval of length $r$ contains at least one element of $E$. The smallest integer $r$ with this property is called the \emph{syndeticity constant} of $E$.
\end{definition}
\begin{definition}
Let $r\geq 2$.
A set $S$ of integers is \emph{$r$-syndetic large} if every syndetic set with syndeticity constant less than or equal to $r$ contains arbitrarily long arithmetic progressions with step in $S$.
\end{definition}

The following proposition is a finite version of the equivalence between characterizations~\eqref{it:largel} and~\eqref{it:intersectivel} of multiple recurrence in Theorem~\ref{th:mult-top-recur}:
\begin{proposition}
\begin{enumerate}
\item
\label{it:obvious}
Every $r$-large set is $r$-syndetic large.
\item 
\label{it:slarge-large}
Every $(2r-1)$-syndetic large set is $r$-large.
\end{enumerate}
\end{proposition}
\begin{proof} 
By using a cover of $\N$ obtained by translates of an $r$-large set $S$ and the associated partition of $\N$, the first statement 
follows.  

For the second statement, assume that $S$ is $(2r-1)$-syndetic large set.
Let $\ell\geq 2$ be an integer and let   $\N=C_1\cup\dots\cup C_r$ be a $r$-coloring of $\N$.  We want to build a monochromatic progression of length $\ell$ and step in $S$. 

Define $E\subset \N$ such that for $n>0$  and $1\leq i\leq r,\ rn+i\in E$ 
if and only if $n\in C_i$.
Then each subinterval of $\N$ of the form $(nr,(n+1)r]$ contains a unique point of $E$, and the congruence class modulo $r$ of this integer is given by the color of $n$. In particular, the difference between two consecutive points of $E$ is $\leq 2r-1$, and $E$ is  syndetic with syndeticity constant $\leq 2r-1$.

Since $S$ is $(2r-1)$-syndetic large,
 $E$ contains an arithmetic progression
$\{a, a+n,\dots,a+(\ell r-1)n\}$ of length $\ell r$ and step $n\in S$. Thus $E$ also  contains the sub-progression
$\{a, a+rn,\dots,a+(\ell-1)rn\}$
of length $\ell$ and step $rn\in rA$. Write $a=rb+i$ where $b\geq 0$ and $1\leq i\leq r$ and rewrite this sub-progression as 
$$
\bigl\{(b+jn)r+i\colon 0\leq j<\ell\bigr\}.
$$
By definition of $E$, all the integers $b+jn$, $0\leq j<\ell$, belong to $C_i$. They form a monochromatic progression of length $\ell$ for the initial coloring  with step in $S$.
\end{proof}

\end{document}